\documentclass[12pt]{amsart}
\usepackage{}
\usepackage{amsmath}
\usepackage{amsfonts}
\usepackage{amssymb}
\usepackage[all,cmtip]{xy}           %xypic macro for latex2.09
\usepackage{bbm}
\usepackage{bbding}
\usepackage{txfonts}
\usepackage[shortlabels]{enumitem}
\usepackage{ifpdf}
\ifpdf
\usepackage[colorlinks,final,backref=page,hyperindex]{hyperref}
\else
\usepackage[colorlinks,final,backref=page,hyperindex,hypertex]{hyperref}
\fi
\usepackage{amscd}
\usepackage{tikz-cd}
\usepackage[active]{srcltx}

\makeatletter

\newtheorem{thm}{Theorem}[section]
\newtheorem{prop}[thm]{Proposition}
\newtheorem{lem}[thm]{Lemma}

\newtheorem{prop-def}{Proposition-Definition}[section]
\theoremstyle{definition}
\newtheorem{defn}[thm]{Definition}

\newtheorem{remark}[thm]{Remark}
\newtheorem{exam}[thm]{Example}
\newtheorem{Ques}[thm]{Question}

\newtheorem{Conj}[thm]{Conjecture}

\newcommand{\nc}{\newcommand}

\nc{\delete}[1]{{}}
\nc{\mmargin}[1]{}
\nc{\mlabel}[1]{\label{#1}}  % Use this to suppress names
\nc{\mcite}[1]{\cite{#1}}  % Use this to suppress names
\nc{\mref}[1]{\ref{#1}}  % Use this to suppress names
\nc{\mbibitem}[1]{\bibitem{#1}} % Use this to show number
\delete{
	\nc{\mlabel}[1]{\label{#1}  % Use the next two lines to show names
		{\hfill \hspace{1cm}{\bf{{\ }\hfill(#1)}}}}
	\nc{\mcite}[1]{\cite{#1}{{\bf{{\ }(#1)}}}}  % Use this lines to show names
	\nc{\mref}[1]{\ref{#1}{{\bf{{\ }(#1)}}}}  % Use this lines to show names
	\nc{\mbibitem}[1]{\bibitem[\bf #1]{#1}} % Use this to show name
}

 \font\cyrs=wncyr7

\newcommand{\bk}{{\mathbf{k}}}
\nc{\vep}{\varepsilon}
\nc{\bin}[2]{ (_{\stackrel{\scs{#1}}{\scs{#2}}})}  %binomial coeff
\nc{\binc}[2]{(\!\! \begin{array}{c} \scs{#1}\\
		\scs{#2} \end{array}\!\!)}  %binomial coeff
\nc{\bincc}[2]{  ( {\scs{#1} \atop
		\vspace{-1cm}\scs{#2}} )}  %binomial coeff
\nc{\oline}[1]{\overline{#1}}
\nc{\mapm}[1]{\lfloor\!|{#1}|\!\rfloor}
\nc{\bs}{\bar{S}}
\nc{\la}{\longrightarrow}
\nc{\ot}{\otimes}
\nc{\rar}{\rightarrow}
\nc{\lon }{\,\rightarrow\,}
\nc{\dar}{\downarrow}
\nc{\dap}[1]{\downarrow \rlap{$\scriptstyle{#1}$}}
\nc{\defeq}{\stackrel{\rm def}{=}}
\nc{\dis}[1]{\displaystyle{#1}}
\nc{\dotcup}{\ \displaystyle{\bigcup^\bullet}\ }
\nc{\hcm}{\ \hat{,}\ }
\nc{\hts}{\hat{\otimes}}
\nc{\hcirc}{\hat{\circ}}
\nc{\lleft}{[}
\nc{\lright}{]}
\nc{\curlyl}{\left \{ \begin{array}{c} {} \\ {} \end{array}
	\right .  \!\!\!\!\!\!\!}
\nc{\curlyr}{ \!\!\!\!\!\!\!
	\left . \begin{array}{c} {} \\ {} \end{array}
	\right \} }
\nc{\longmid}{\left | \begin{array}{c} {} \\ {} \end{array}
	\right . \!\!\!\!\!\!\!}
\nc{\ora}[1]{\stackrel{#1}{\rar}}
\nc{\ola}[1]{\stackrel{#1}{\la}}%${\Bbb Z}$
\nc{\scs}[1]{\scriptstyle{#1}} \nc{\mrm}[1]{{\rm #1}}
\nc{\dirlim}{\displaystyle{\lim_{\longrightarrow}}\,}
\nc{\invlim}{\displaystyle{\lim_{\longleftarrow}}\,}
\nc{\dislim}[1]{\displaystyle{\lim_{#1}}} \nc{\colim}{\mrm{colim}}
\nc{\mvp}{\vspace{0.3cm}} \nc{\tk}{^{(k)}} \nc{\tp}{^\prime}
\nc{\ttp}{^{\prime\prime}} \nc{\svp}{\vspace{2cm}}
\nc{\vp}{\vspace{8cm}}
\nc{\modg}[1]{\!<\!\!{#1}\!\!>}
\nc{\intg}[1]{F_C(#1)}
\nc{\lmodg}{\!<\!\!}
\nc{\rmodg}{\!\!>\!}
\nc{\cpi}{\widehat{\Pi}}
\nc{\ssha}{{\mbox{\cyrs X}}} %sha as product
\nc{\tsha}{{\mbox{\cyrt X}}}
\nc{\shpr}{\diamond}    %Shuffle product
\nc{\labs}{\mid\!}
\nc{\rabs}{\!\mid}
 \nc{\zhx}{\text{-}}

\nc{\ad}{\mrm{ad}}
\nc{\ann}{\mrm{ann}}
\nc{\Aut}{\mrm{Aut}}
\nc{\Av}{\mrm{Av}}
\nc{\bim}{\mbox{-}\mathsf{Bimod}}
\nc{\br}{\mrm{bre}}
\nc{\can}{\mrm{can}}
\nc{\Cont}{\mrm{Cont}}
\nc{\rchar}{\mrm{char}}
\nc{\cok}{\mrm{coker}}
\nc{\db}{\mrm{db}}
\nc{\de}{\mrm{dep}}
\nc{\dgg}{\mrm{dgg}}
\nc{\dgp}{\mrm{dgp}}
\nc{\dgx}{\mrm{dgx}}
\nc{\Dif}{\mrm{Diff}}
\nc{\dtf}{{R-{\rm tf}}}
\nc{\dtor}{{R-{\rm tor}}}

\nc{\Div}{{\mrm Div}}
\nc{\Diff}{\mrm{DA}}
\nc{\Diffl}{\mathsf{DA}_\lambda}
\nc{\diffo}{{\mathsf{DO}_\lambda}}
\nc{\dl}{{\mathrm{dl}}}
\nc{\alg}{\mathsf{Alg}}
\nc{\End}{\mrm{End}}
\nc{\Ext}{\mrm{Ext}}
\nc{\Fil}{\mrm{Fil}}
\nc{\Fr}{\mrm{Fr}}
\nc{\Frob}{\mrm{Frob}}
\nc{\Gal}{\mrm{Gal}}
\nc{\GL}{\mrm{GL}}
\nc{\Hom}{\mrm{Hom}}
\nc{\Hoch}{\mrm{Hoch}}
\nc{\hsr}{\mrm{H}}
\nc{\hpol}{\mrm{HP}}
\nc{\id}{\mrm{id}}
\nc{\im}{\mrm{im}}
\nc{\Id}{\mrm{Id}}
\nc{\ID}{\mrm{ID}}
\nc{\Irr}{\mrm{Irr}}
\nc{\incl}{\mrm{incl}}
\nc{\length}{\mrm{length}}
\nc{\NLSW}{\mrm{NLSW}}
\nc{\Lie}{\mrm{Lie}}
\nc{\Nij}{\mrm{Nij}}
\nc{\mchar}{\rm char}
\nc{\mpart}{\mrm{part}}
\nc{\ql}{{\QQ_\ell}}
\nc{\qp}{{\QQ_p}}
\nc{\rank}{\mrm{rank}}
\nc{\rcot}{\mrm{cot}}
\nc{\rdef}{\mrm{def}}
\nc{\rdiv}{{\rm div}}
\nc{\Rey}{\mrm{Rey}}
\nc{\rtf}{{\rm tf}}
\nc{\rtor}{{\rm tor}}
\nc{\res}{\mrm{res}}
\nc{\SL}{\mrm{SL}}
\nc{\Spec}{\mrm{Spec}}
\nc{\tor}{\mrm{tor}}
\nc{\Tr}{\mrm{Tr}}
\nc{\tr}{\mrm{tr}}
\nc{\wt}{\mrm{wt}}
\def\ot{\otimes}

\nc{\udl}{{\mathrm{udl}}}

\nc{\bfk}{{\bf k}}
\nc{\bfone}{{\bf 1}}
\nc{\bfzero}{{\bf 0}}
\nc{\detail}{\marginpar{\bf More detail}
	\noindent{\bf Need more detail!}
	\svp}
\nc{\gap}{\marginpar{\bf Incomplete}\noindent{\bf Incomplete!!}
	\svp}
\nc{\FMod}{\mathbf{FMod}}
\nc{\Int}{\mathbf{Int}}
\nc{\remarks}{\noindent{\bf Remarks: }}
\nc{\ob}{\mathsf{Ob}}

\nc{\BA}{{\mathbb A}}   \nc{\CC}{{\mathbb C}}
\nc{\DD}{{\mathbb D}}   \nc{\EE}{{\mathbb E}}
\nc{\FF}{{\mathbb F}}   \nc{\GG}{{\mathbb G}}
\nc{\HH}{{\mathbb H}}   \nc{\LL}{{\mathbb L}}
\nc{\NN}{{\mathbb N}}   \nc{\PP}{{\mathbb P}}
\nc{\QQ}{{\mathbb Q}}   \nc{\RR}{{\mathbb R}}
\nc{\TT}{{\mathbb T}}   \nc{\VV}{{\mathbb V}}
\nc{\ZZ}{{\mathbb Z}}   \nc{\TP}{\widetilde{P}}

\nc{\cala}{{\mathcal A}}    \nc{\calc}{{\mathcal C}}
\nc{\cald}{\mathcal{D}}     \nc{\cale}{{\mathcal E}}
\nc{\calf}{{\mathcal F}}    \nc{\calg}{{\mathcal G}}
\nc{\calh}{{\mathcal H}}    \nc{\cali}{{\mathcal I}}
\nc{\call}{{\mathcal L}}    \nc{\calm}{{\mathcal M}}
\nc{\caln}{{\mathcal N}}    \nc{\calo}{{\mathcal O}}
\nc{\calp}{{\mathcal P}}    \nc{\calr}{{\mathcal R}}
\nc{\cals}{{\mathcal S}}    \nc{\calt}{{\Omega}}
\nc{\calv}{{\mathcal V}}    \nc{\calw}{{\mathcal W}}
\nc{\calx}{{\mathcal X}}    \nc{\calu}{{\mathcal U}}
\nc{\caly}{{\mathcal Y}}

\nc{\uOpAlg}{\mathfrak{uOpAlg}}

\nc{\OpAlg}{\mathfrak{OpAlg}}

\nc{\ComOpAlg}{\mathfrak{ComOpAlg}}

\nc{\OpVect}{\mathfrak{OpVect}}

\nc{\OpSet}{\mathfrak{OpSet}}

\nc{\OpMon}{\mathfrak{OpMon}}

\nc{\ComOpMon}{\mathfrak{ComOpMon}}

\nc{\OpSem}{\mathfrak{OpSem}}

\nc{\ComOpSem}{\mathfrak{ComOpSem}}

\nc{\uAlg}{\mathfrak{uAlg}}

\nc{\Alg}{\mathfrak{Alg}}

\nc{\ComAlg}{\mathfrak{ComAlg}}

\nc{\Vect}{\mathfrak{Vect}}

\nc{\Set}{\mathfrak{Set}}

\nc{\Mon}{\mathfrak{Mon}}

\nc{\ComMon}{\mathfrak{ComMon}}

\nc{\Sem}{\mathfrak{Sem}}

\nc{\ComSem}{\mathfrak{ComSem}}
\nc{\fraka}{{\mathfrak a}}
\nc{\frakb}{\mathfrak{b}}
\nc{\frakg}{{\frak g}}
\nc{\frakl}{{\frak l}}
\nc{\fraks}{{\frak s}}
\nc{\frakB}{{\frak B}}
\nc{\frakm}{{\frak m}}
\nc{\frakM}{{\mathfrak M}}
\nc{\frakp}{{\frak p}}
\nc{\frakW}{{\frak W}}
\nc{\frakX}{{\frak X}}
\nc{\frakS}{{\frak S}}
\nc{\frakA}{{\frak A}}
\nc{\frakx}{{\frakx}}

\nc{\lir}[1]{\textcolor{red}{\underline{Li:}#1 }}

\begin{document}

\title[differential type]{Gr\"obner-Shirshov bases and linear bases  for  free differential type algebras over algebras}

\author{Zihao Qi, Yufei Qin and Guodong Zhou}
\address{Zihao  Qi, School of Mathematical Sciences,
	Fudan University, Shanghai 200433, China}
\email{qizihao@foxmail.com}

\address{Yufei Qin and Guodong Zhou,  School of Mathematical Sciences, Shanghai Key Laboratory of PMMP,
	East China Normal University,
	Shanghai 200241,
	China}
\email{290673049@qq.com}
 \email{gdzhou@math.ecnu.edu.cn}

\date{\today}

\begin{abstract} We study  a question which can be roughly stated as follows: Given a (unital or nonunital)  algebra $A$ together   with a   Gr\"obner-Shirshov basis $G$,
consider the free operated algebra $B$ over $A$, such that the  operator satisfies some polynomial identities $\Phi$ which are Gr\"obner-Shirshov in the sense of Guo et al., when does
  the union $\Phi\cup G$ will be  an operated  Gr\"obner-Shirshov  basis  for $B$?  We answer this question in the affirmative under a mild condition in our previous work with Wang.  When this   condition is satisfied, $\Phi\cup G$  is  an operated  Gr\"obner-Shirshov  basis  for $    B$ and as a consequence, we also get  a linear basis of $B$. However,  the condition could not be applied directly to differential type algebras introduced   by   Guo, Sit and Zhang, including usual differential algebras.

This paper  solves completely this problem for     differential type algebras.
 Some new monomial orders are introduced which, together with some known ones,    permit the application of the previous result to most of differential type algebras,  thus providing    new operated GS bases and   linear bases for these differential type algebras.
  Versions    are presented both  for unital and nonunital algebras.
However, a class of examples are also presented, for which the natural expectation in the  question is wrong and these examples are  dealt with  by direct inspection.
\end{abstract}

\subjclass[2000]{
13P10 %Gr?bner bases; other bases for ideals and modules (e.g., Janet and border bases)
03C05 %Equational classes, universal algebra
08B20 % Free algebras
12H05 % Differential algebra
16S10 % Rings determined by universal properties (free algebras, coproducts, adjunction of inverses, etc.)
%!7B38  %Yang-Baxter equations and Rota-Baxter operators
}

\keywords{Differential type OPIs;  Gr\"obner-Shirshov bases; operated algebras; free operated algebras over algebras; operated polynomial identities}

\maketitle

 \tableofcontents

\allowdisplaybreaks

\section*{Introduction}

In our previous work with Wang \cite{QQWZ}, we studied operated Gr\"obner-Shirshov (aka GS) bases  for  free  operated  algebras over associative algebras whose operator satisfies some polynomial identities.
This paper  continues the work of \cite{QQWZ} and concentrates into    differential type algebras.

\subsection{From usual GS bases to operated GS bases}\

 Gr\"obner-Shirshov basis theory was invented by Shirshov \cite{Shirshov} and Buchberger \cite{Buc} in the  sixties  of last century. It becomes one of the main tools of computational algebra since then; see for instance \cite{Green, BokutChen14, BokutChen20}. In order to deal with algebras endowed with operators, Guo et al.  introduced   a  GS basis theory in operated contexts in a series of papers \cite{Guo09a,  GGSZ14, GSZ13,  GaoGuo17} (see also \cite{BokutChenQiu})
 with the goal to attack Rota's program \cite{Rota} to classify ``interesting'' operators on algebras. For a global view about the state of art, we refer the reader to the survey paper \cite{GGZ21} and for recent development, see \cite{ZhangGao18, ZhangGao19, CLZZ20, ZhangGaoGuo21,  GuoLi21,  QQWZ, ZhuZhangGao21, ZhangGaoGuo}.

Guo et al.  considered operators satisfying some polynomial identities, hence called operated polynomial identities (aka. OPIs). Via GS basis theory and  the somewhat equivalent theory: rewriting systems, they could define when OPIs are GS. They are mainly interested into two classes of OPIs:  differential type OPIs and Rota-Baxter type OPIs, which are carefully  studied  in \cite{GSZ13, GGSZ14, GaoGuo17}.

In this paper we are interested into  differential type OPIs and  differential type algebras.

%\subsection{Differential type algebras and their GS bases}

 \subsection{Free  operated algebras over algebras}\

Recently, there is a need to develop free  operated algebras satisfying some OPIs  over a fixed  algebras and construct   GS  bases and linear bases  for these free  algebras as long as a GS  basis is known for the given algebra.   Ebrahimi-Fard and Guo  \cite{ERG08a} used rooted trees and forests to give explicit constructions of free noncommutative
Rota-Baxter algebras on modules and sets; Lei and Guo \cite{LeiGuo} constructed the linear basis of free Nijenhuis algebras over associative algebras;
   Guo and Li \cite{GuoLi21} gave a linear basis of the free differential algebra over associative algebras by introducing the notion of differential GS bases.

As in \cite{QQWZ}, we want to consider a  question which can be roughly stated as follows:
\begin{Ques} \label{Ques: GS basis for free algebras over algebras} Given a (unital or nonunital)  algebra  $A$  with a   GS basis  $G$ and  a set $\Phi$ of OPIs,
	  assume that   these OPIs  $\Phi$ are GS  in the sense of \cite{BokutChenQiu, GSZ13, GGSZ14,GaoGuo17}.  Let $B$
be the free  operated algebra satisfying $\Phi$ over $A$.  Under what conditions, $\Phi\cup G$ will be  a GS  basis  for $    B$?
\end{Ques}
We answer this question in the affirmative under a mild condition in \cite[Theorem 5.9]{QQWZ}.  When this   condition is satisfied, $\Phi\cup G$  is  a GS  basis  for $    B$ and as a consequence, we also get  a linear basis of $B$. This result    has been applied to all Rota-Baxter type OPIs, a class of differential type OPIs, averaging OPIs and Reynolds OPI in \cite{QQWZ}.

However, this result can NOT be applied directly to all differential type OPIs. The main reason is that under the chosen monomial order for differential type OPIs, the condition of \cite[Theorem 5.9]{QQWZ} is not satisfied.

The goal of this paper is to introduce some new monomial orders which permit to deal with nearly all free  differential type algebras over algebras.

Moreover, we provide  a class of  counter-examples for which  the expectation in Question~\ref{Ques: GS basis for free algebras over algebras} is wrong; see Theorems~\ref{thm:GS basis of nonunital phi4} and  \ref{Cor: linear basis for type phi4}.

\subsection{Differential type algebras over algebras}\

A commutative algebra equipped with an operator satisfying the Leibniz rule, is called a differential algebra. The notion is originated by  Ritt \cite{Ritt34,Ritt50}  in the algebraic study of differential equations in the 1930s. This mathematical subject has been developed into a broad realm such as differential Galois theory  \cite{Magid, vdPS}, differential algebraic geometry and differential algebraic groups \cite{Kolchin}. In recent years, researchers began to investigate noncommutative differential algebras  in order to broaden the scope of the theory to  include   path algebras, for instance,  and to have a more meaningful differential Lie algebra theory  \cite{Poi, Poi1} and also from  an operadic point of view \cite{Loday, DL}.
There are also some recent work dealing with other algebraic structures endowed with derivations \cite{TFS, Das, Xu}.

As mentioned above, \cite[Theorem 5.9]{QQWZ}  can NOT be applied directly to all differential type OPIs in order to obtain  GS  bases and linear bases for free differential type algebras over algebras.
In this paper, we introduce some new monomial orders and solve completely this question.

 Our method is as follows: We change a differential type OPI to an equivalent one in the sense of Definition~\ref{Def: equivalent OPIs};   by choosing convenient monomial orders,    we can apply \cite[Theorem 5.9]{QQWZ} to obtain a GS basis;    as a consequence, we get a linear basis.
 We also deal with the problem in the nonunital case; see Theorem~\ref{Thm: GS basis for free nonunital Phi algebra over nonunital alg}.
Quite interestingly, there are some exceptional cases for which our method could not apply and we treat these cases directly.

We would like to mention that Guo and Li  introduced the notion of   differential GS bases to deal with the above problem in \cite{GuoLi21} and
their method is completely different from ours; see Remark~\ref{Rem: Guo and Li's work}.

\subsection{Outline of the paper}\

This paper is organized as follows. In Section 1   some basic definitions  and  constructions of free objects in operated contexts are recalled and some new monomial orders are introduced.  Section 2 contains an account of  the theory of operated GS bases,   which serve  as a main tool in this paper; the main result of  \cite{QQWZ} is recalled and a nonunital version is presented; a remainder about the basic theory of rewriting systems.  In Section 3,  we consider differential type OPIs and they are replaced by some equivalent ones.   Section 4   is devoted to  operated GS  bases and linear bases  for  free  nonunital differential type algebras  over an algebra, while the last section considers unital ones.

\medskip

\textbf{Notation:} Throughout this paper, $\bk$ denotes a  base field. All the vector spaces and algebras are over $\bk$.

\bigskip

\section{New monomial orders on  free operated semigroups and monoids }

In this section,  we recall the construction of free operated semigroups and monoids due to Guo \cite{Guo09a}.
We will  define a new monomial  order  $\leq_{\dl} $  on free operated semigroups  by combining several well orders and extend it to a monomial  order    $\leq_{\udl}$   on  free operated   monoids via Remark~\ref{Rem: another construction of m(X)}.  The  main results of this paper will  highly depend on these two new monomial  orders.

\subsection{Free operated semigroups and monoids}

\begin{defn} [{\cite{Guo09a, GGSZ14, GSZ13,  GaoGuo17}}]\label{Def: operated vector spaces to unital algebras}
An operated semigroup is a semigroup $S$ with a map $P:S\rightarrow S$ (which is not necessarily a homomorphism of semigroups).
Operated monoids can be defined similarly.
	For  operated $\bk$-algebra (resp. unital operated $\bk$-algebra), we ask additionally that the operator should be a $\bk$-linear map. %is a  $\bk$-algebra (resp. unital $\bk$-algebra)	  $A$  together with a $\bk$-linear map $ P_{A}: A\rightarrow A $.% A morphism between   operated $\bk$-algebras (resp. unital operated $\bk$-algebras) $ (A, P_A) $ and  $(B, P_B)  $ is a homomorphism of  $\bk$-algebras (resp. unital $\bk$-algebras) $ f : A \rightarrow B $ such that $f\circ P_A = P_B\circ f $.
	
	% as the following commutation diagram.
	%\begin{equation}\label{key}
	%	\nonumber\xymatrix{A\ar[r]^{f}\ar[d]^{P_A} &B\ar[d]^{P_B}\\A\ar[r]^{f}&B}
	%\end{equation}	
	
We use $ \OpAlg $ (resp.  $ \uOpAlg $) to denote the category of operated $\bk$-algebras (resp. unital operated $\bk$-algebras) with obvious morphisms.
\end{defn}

Throughout this section,   $X$ denotes  a set.  Denote by $\bk X$ (resp.  $\cals(X)$,  $\calm(X)$) the free $\bk$-space (resp. free semigroup, free monoid) generated by $X$.

\medskip

 We   recall the construction of the free operated semigroup over a given set $X$ due to  Guo   \cite{Guo09a}.

  For a set $X$, denote by  $\lfloor X \rfloor$     the set of the formal elements $ \lfloor x \rfloor, x\in X$. The inclusion   $ X\hookrightarrow X\sqcup  \left \lfloor X\right \rfloor$,  which identifies $X$ with the first component, induces an injective  semigroup homomorphism
$$ i_{0,1}: \frakS_0(X):=\cals(X)  \hookrightarrow \frakS_1(X):= \cals(X\sqcup  \left \lfloor X\right \rfloor).
$$

For $ n\geq2 $, assume  that we have constructed $\frakS_{n-2}(X)$ and  $\frakS_{n-1}(X)= \cals(X\sqcup  \left \lfloor \frakS_{n-2}(X)\right \rfloor) $  endowed  with  an  injective homomorphism of semigroups
 $
i_{n-2,n-1}: \frakS_{n-2}(X) \hookrightarrow \frakS_{n-1}(X).
$
We define the semigroup
$$\frakS_{n}(X):= \cals(X\sqcup  \left \lfloor \frakS_{n-1}(X)\right \rfloor)
$$ and
the natural  injection
$$
\Id_X\sqcup \left \lfloor i_{n-2,n-1}\right \rfloor :X\sqcup\left \lfloor \frakS_{n-2}(X)\right \rfloor\hookrightarrow X\sqcup\left \lfloor \frakS_{n-1}(X)\right \rfloor
$$
induces an injective semigroup  homomorphism
$$	i_{n-1,n}: \frakS_{n-1}(X)= \cals(X\sqcup  \left \lfloor \frakS_{n-2}(X)\right \rfloor) \hookrightarrow  \frakS_{n}(X) = \cals(X\sqcup  \left \lfloor \frakS_{n-1}(X)\right \rfloor).$$
Define $  \frakS(X)=\varinjlim \frakS_{n}(X) $  and the map  sending  $u\in \frakS_n(X)$ to $\left \lfloor u\right \rfloor\in \frakS_{n+1}(X)$ induces an  operator $P_{\frakS(X)}$  on $\frakS(X)$. % By the limit construction, there are injective semigroup homomorphisms    $ i_n:\frakS_{n}(X)\hookrightarrow\frakS(X),  n\geq 0$ and $i:X\hookrightarrow\frakS(X)$.

 The construction of the free operated monoid $\frakM(X)$ over a set $X$ is similar, by just replacing $\cals$ by $\calm$ everywhere in the construction.

 \begin{remark}\label{Rem: another construction of m(X)}  We will use another  construction of  $\frakM(X)$.  In fact, add a symbol $\lfloor1\rfloor$ to $X$ and form $ \frakS(X \bigsqcup \lfloor1\rfloor)$, then $\frakM(X)$ can be obtained from $\frakS(X \bigsqcup \lfloor1\rfloor)$ by just adding the empty word $1$.

 \end{remark}

 It is easy to see that $\bk\frakS(X)$ (resp. $\bk\frakM(X)$) is the free nonunital (resp, unital) operated algebra generated by $X$.
 More constructions of free objects in operated contexts can be found in \cite{QQWZ}.

\subsection{Monomial orders}\

We need some preliminaries about orders.
\begin{defn}
	Let $X$ be a nonempty set.
	\begin{itemize}
		\item [(a)] A preorder   $\leq$  is a binary relation on $X$ that is reflexive and transitive, that is, for all $x, y, z \in X$, we have
		\begin{itemize}
			\item [(i)] $x \leq x$; and
			\item[(ii)]if $x \leq y, y \leq z$, then $x \leq z$. We denote $x=_{X} y$ if $x \leq y$ and $x \geq y$. If $x \leq y$ but $x \neq y$, we write $x< y$ or $y> x$.
		\end{itemize}
		\item[(b)] A pre-linear order $\leq$ on $X$ is a preorder $\leq$ such that either $x \leq y$ or $x \geq y$ for all $x, y \in X$.
		\item[(c)] A linear order or a total order $\leq$ on $X$ is a pre-linear order $\leq$  such that  $\leq$ is symmetric, that is,   $x \leq y$ and $y \leq x$ imply    $x=_X y$.

\item[(d)] A preorder   $\leq$ on $X$ satisfies the descending chain condition, if $x_1\geq x_2\geq x_3\geq \cdots$, then there exists $N\geq 1$ such that $x_N=_X x_{N+1}=_X\cdots$.
    A linear order satisfying the descending chain condition is called a well order.

	\end{itemize}
\end{defn}

Before giving the definition of monomial orders, we need to introduce the following notions.
\begin{defn} [\cite{BokutChenQiu, GSZ13, GGSZ14, GaoGuo17}]
	Let $Z$ be a set and $\star$ a symbol not in $Z$.
	\begin{itemize}
		\item [(a)] Define $\frakM^\star(Z)$  to be the subset of $\frakM(Z\cup\star)$ consisting of elements with $\star$ occurring only once.
		\item [(b)] For $q\in\frakM^\star(Z)$ and $u\in   \frakM(Z)$,	we define $q|_{u}\in \frakM(Z)$ obtained by
		replacing the symbol $\star$ in $q$ by $u$. In this case, we say $u$ is a subword of $q|_{u}$.
		\item [(c)] For $q\in\frakM^\star(Z)$ and $s=\sum_ic_iu_i\in \bk \frakM(Z)$  with $c_i\in\bk$ and $u_i\in\frakM(Z)$, we define
		$$q|_s:=\sum_ic_iq|_{u_i}.$$
		\item [(d)] Define $\frakS^\star(Z)$ to be the subset of $\frakS(Z\cup\star)$ consisting of elements with $\star$ occurring only once. It is easy to see $\frakS^\star(Z)$ is a subset of $\frakM^\star(Z)$, so we also have the above definitions for $\frakS^\star(Z)$ by restriction.
	\end{itemize}
\end{defn}

\begin{defn} [\cite{BokutChenQiu, GSZ13, GGSZ14, GaoGuo17}]
	Let $Z$ be a set. We deonte $u< v $ if $u\leq v$ but $u\neq v$ for an order $\leq$.	
	
	\begin{itemize}	
		\item [(a)]A monomial order on $\cals(Z)$ is a well-order $\leq$ on $\cals(Z)$ such that
		$$  u < v \Rightarrow  uz < vz\text{ and }wu<wv  \text{ for any }u, v, w,z\in \cals(Z);$$
		a monomial order on $\calm(Z)$ is a well-order $\leq$ on $\calm(Z)$ such that
		$$ u < v \Rightarrow  wuz < wvz \text{ for any }u, v, w,z\in \calm(Z).$$
		%	(	e.g the  deg-lex order, see Definition \ref{Def: deg-lex order}).
		
		\item [(b)]A monomial order on  $\frakS(Z)$  is a well-order $\leq$ on $\frakS(Z)$  such that
		$$u< v \Rightarrow q|_u<q|_v\quad\text{for all }u,v\in\frakS(Z)\text{ and } q\in \frakS^\star(Z);$$
		a monomial order on $\frakM(Z)$ is a well-order $\leq$ on $\frakM(Z)$ such that
		$$u< v \Rightarrow q|_u<q|_v\quad\text{for all }u,v\in\frakM(Z)\text{ and } q\in \frakM^\star(Z). $$
	\end{itemize}
\end{defn}

\begin{defn} 	\begin{itemize}
		\item [(a)] Given some preorders $\leq_{\alpha_{1}}, \dots, \leq_{\alpha_{k}}$ on a set $X$ with $k\geq 2$, introduce another preorder
$\leq_{\alpha_{1}, \dots, \alpha_{k}}$ by imposing recursively
$$
	u \leq_{\alpha_{1}, \dots, \alpha_{k}} v \Leftrightarrow\left\{
	\begin{array}{l}
		u<_{\alpha_{1}} v,  \text{ or }  \\
		 u=_{\alpha_{1}} v \text { and } u \leq_{\alpha_{2}, \dots, \alpha_{k}} v.
	\end{array}\right.
	$$
	%where $\leq_{\beta}:=\leq_{\alpha_{2}, \dots, \alpha_{k}}$ is defined by   induction.
%, with the convention that $\leq_{\beta}$ is the trivial relation when $k=1$, namely $u \leq_{\beta} v$ for all $u, v \in X$.
\item [(b)] 	Let $k \geq2$ and let $\leq_{\alpha_i}$ be a pre-linear order on $X_{i},~1 \leq i \leq k$. Define the lexicographical product order $\leq_{\text{clex}}$ on the cartesian product $X_{1} \times X_{2} \times \cdots \times X_{k}$ by recursively defining
$$(x_{1},\cdots, x_{k}) \leq_{\text {clex}}(y_{1},\cdots, y_{k}) \Leftrightarrow\left\{\begin{array}{l}x_{1}<_{\alpha_{1}} y_{1}, \text {or } \\  x_{1}=_{X_{1}}y_{1} \text { and }\left(x_{2}, \cdots, x_{k}\right) \leq_{\rm{clex}}\left(y_{2}, \cdots, y_{k}\right),\end{array}\right.$$ where $\left(x_{2}, \cdots, x_{k}\right) \leq_{\rm{clex}}\left(y_{2}, \cdots, y_{k}\right)$ is defined by induction, with the convention that $\leq_{\rm{clex}}$ is the trivial relation when $k=1$.
\end{itemize}
\end{defn}

By the proof of \cite[Lemma~5.4(a)]{GGSZ14}, we have the following result:
\begin{lem} \label{sequence of order gives linear order}
	\begin{itemize}
		\item[(a)]	Let $k \geq  2 .$ Let $\leq_{\alpha_{1}}, \dots, \leq_{\alpha_{k-1}}$ be pre-linear orders on $X$, and $\leq_{\alpha_{k}}$ a linear order on $X$. Then   $\leq_{\alpha_{1}, \dots, \alpha_{k}}$ is a linear order on $ X$.
		\item[(b)] Let $ \leq_{\alpha_{i}} $ be a well order on $X_i$, $1\leq i\leq k$.  Then the lexicographical product
		order $\leq_{\rm{clex}}$ is a well order on the cartesian product $X_1\times X_2\times \cdots \times X_k$.
	
	\end{itemize}
 	%If moreover, all$\leq_{\alpha_{1}}, \dots, \leq_{\alpha_{k}}$ satisfy the descending chain condition, then $\leq_{\alpha_{1}, \dots, \alpha_{k}}$  is also a well order.
\end{lem}

\medskip

Let us first recall the well known degree lexicographical order.
\begin{defn}\label{Def: deg-lex order}
	Let $X$ be a set endowed with a well order $\leq_{X}$.   For $u=u_{1} \cdots u_{r} \in \cals(X) $ with $u_{1}, \dots, u_{r} \in X,$ define $\operatorname{deg}_{X}(u)=r $.
	
	Introduce  the degree lexicographical order $\leq_{\rm {dlex }}$ on $\cals(X)$ by imposing, for any $u\neq  v \in \cals(X)$, $u <_{\rm {dlex}}v$   if
	\begin{itemize}
		\item[(a)] either $\operatorname{deg}_{X}(u)<\operatorname{deg}_{X}(v)$, or
		\item[(b)] $\operatorname{deg}_{X}(u)=\operatorname{deg}_{X}(v)$, and $u=mu_{i}n$, $v=mv_{i}n^\prime$ for some $m,n,n^\prime\in \cals(X)$ and $u_{i},v_{i}\in X$ with $u_{i}<_{X} v_{i}$.
	\end{itemize}
\end{defn}
It  is obvious that the degree lexicographic order $\leq_{\mathrm{dlex}}$ on $\cals(X)$ is a well order and  $\leq_{\mathrm{dlex}}$ can be extended to  $\calm(X)$ by setting $1$ to be the least element.

For two well ordered sets $X$ and $Y$, one obtains      an extended well order on the disjoint union $X\sqcup Y$  by defining $a<b$ for all $a\in X$ and $b\in Y$.
%;see \cite[Lemma~5.1]{GSX13}

  We now define a linear order on $\frakS(X)$, by the following recursion:
\begin{itemize}
	\item [(a)]  Let $u,v\in \frakS_{\operatorname{Dlex}_0}(X)=\cals(X)$.   We define
	$$u\leq_{\operatorname{Dlex}_0} v \Leftrightarrow u \leq_{\mathrm{dlex}}v.$$
	\item [(b)] Assume that we have constructed  a well order $\leq_{\operatorname{Dlex}_n}$  on $\frakS_n(X)$ for $n\geq 0$ extending all $\leq_{\operatorname{Dlex}_i}$ for any $0\leq i\leq n-1$.   The  well order $\leq_{\operatorname{Dlex}_n}$ induces a well order   on $\lfloor\frakS_n(X)\rfloor$ and   by imposing  $a>b$ for all $a\in \lfloor\frakS_n(X)\rfloor$ and $b\in X$, we obtain an extended well order  on  $X\sqcup \lfloor\frakS_n(X)\rfloor$.
Denote by    $\leq_{\operatorname{Dlex}_{n+1}}$ the induced  degree lexicographic order on   $\frakS_{n+1}(X)=\cals(X\sqcup \lfloor\frakS_n(X)\rfloor)$.
\end{itemize}
Obviously     $\leq_{\operatorname{Dlex}_{n+1}}$ extends    $\leq_{\operatorname{Dlex}_{n}}$.  By a limit process, we get a preorder on $\frakS(X)$ which will be  denoted by
 $\leq_{\operatorname{Dlex}}$.

 It is easy to see that this is a linear order, but it is NOT a well order. An example is given by
 $$u\lfloor v\rfloor>\lfloor u\lfloor v\rfloor\rfloor>\lfloor \lfloor u\lfloor v\rfloor\rfloor\rfloor>\cdots$$
 for $u, v\in X$.

Recall that an arbitrary element of $\frakS(X)$ has a unique expression
$u=u_{0}\lfloor u_{1}^{*}\rfloor u_{1}\lfloor u_{2}^{*}\rfloor \cdots\lfloor u_{r}^{*} \rfloor u_{r}$,  where $u_{0}, u_{1}, \dots, u_{r}
	 \in \calm(X)$ and $u_{1}^{*}, u_{2}^{*}, \dots, u_{r}^{*} \in \frakS(X).$ Notice that $u_{0}, u_{1}, \dots, u_{r}$ could be the empty word. Define the $P$-breath $|u|_P$ of $u=u_{0}\lfloor u_{1}^{*}\rfloor u_{1}\lfloor u_{2}^{*}\rfloor \cdots\lfloor u_{r}^{*} \rfloor u_{r}$  to be $r$.
\begin{defn}[{\cite[Definition~5.3]{GGSZ14} }]
	
	For two elements $u, v \in \frakS(X)$,
	\begin{itemize}
		\item [(a)]  define
		$$
		u \leq_{\operatorname{dgp}} v \Leftrightarrow \operatorname{deg}_{P}(u) \leq \operatorname{deg}_{P}(v),
		$$
		where the $P$-degree $\operatorname{deg}_{p}(u)$ of $u$ is the number of occurrence  of $P=\lfloor~\rfloor$ in $u$;
		\item [(b)]   define
		$$
		u \leq_{\operatorname{dgx}} v \Leftrightarrow \operatorname{deg}_{X}(u) \leq \operatorname{deg}_{X}(v),
		$$
		where the $X$-degree $\operatorname{deg}_{X}(u)$ is the number of elements of $X$ occurring in $u$ (including the repetitive ones);
		\item[(c)]   define
			$$
		u \leq_{\operatorname{brp}} v \Leftrightarrow |u|_P\leq |v|_P.
		$$
	\end{itemize}
	
\end{defn}

\begin{defn}

	 	For any $u\in\frakS(X)$, let $u_1,\dots ,u_n\in X$ be  all the elements occurring in $u$ from left to right. If a half bracket $\lfloor$  (resp. $\rfloor$) is between $u_i$ and $u_{i+1}$, where $1\leq i<n$,   the G-degree of this half bracket is defined  to be $i$;  if there is a  half bracket $\lfloor$ (resp. $\rfloor$)   appearing on the left of $u_1$ (resp. on the right of $u_n$), we define the G-degree of this half bracket to be $0$ (resp. $n$). %If the half bracket $\lfloor$  (resp. $\rfloor$) appears , we define the G-degree of this half bracket to be $n$.
		 We denote $\deg_G(u)$ by the sum of the G-degree of all the half brackets in $u$.

For $u,v\in\frakS(X)$, define the G-degree order $\leq_{\mathrm{dgg}}$ by
		$$ u\leq_{\mathrm{dgg}}v\Leftrightarrow \deg_G(u)\leq \deg_G(v).$$

\end{defn}

It is easy to obtain the following lemma whose  proof is thus omitted.

\begin{lem}
	The orders $\leq_{\mathrm{dgp}}$, $\leq_{\mathrm{dgx}}$,  $\leq_{\operatorname{brp}}$ and $\leq_{\mathrm{dgg}}$ are pre-linear orders satisfying the descending chain	condition.
\end{lem}

%
%\begin{proof}
%	 We prove the claim by induction on $k \geq 1$. When $k=1,~ \leq_{\alpha_{1}}$ is a linear order by the assumption. Assume that the claim holds for $k = n$ and consider the case when $k=n+1.$ Denote $\leq_{\beta}=\leq_{\alpha_{2}, \dots, \alpha_{n+1}}$. Then $\leq_{\beta}$ is a linear order by induction.  For all $u, v \in X,$ we have $u \leq_{\alpha_{1}} v$ or $u \geq_{\alpha_{1}} v$ since $\leq_{\alpha_{1}}$ is a pre-linear order. If $u \neq_{\alpha_{1}} v$, then we have $u\leq{\alpha_{1}}v$ or $u\geq{\alpha_{1}} v$. Thus we obtain $u\leq{\alpha_{1} ,\beta} v$ or $u\geq{\alpha_{1}, \beta} v$ and we are done. If $u=_{\alpha_1}v$, then $u \geq_{\beta} v$ or $u \leq_{\beta} v$ since $\leq_{\beta}$ is a linear order. Thus we have $u \geq_{\alpha_{1}, \beta} v$ or $u \leq_{\alpha_{1}, \beta} v$. Therefore, $\leq_{\alpha_{1}, \beta}$ is a linear order.
%\end{proof}

By Lemma~\ref{sequence of order gives linear order}, we can define  a linear order $\leq_{\mathrm{dl}}$ on $\frakS(X)$,
$$u\leq_{\mathrm{dl}}v\Leftrightarrow u\leq_{\mathrm{dgp,dgx,dgg},\operatorname{Dlex}}v\Leftrightarrow \left\{
\begin{array}{lcl}
	u<_{\mathrm{dgp}} v, \text{or }\\
	  u=_{\mathrm{dgp}}v \text{ and }u<_{\dgx} v, \text{or }\\
	 u=_{\mathrm{dgp}}v, u=_{\mathrm{dgx}}v   \text{ and } u<_{\dgg} v, \text{or }\\
	 u=_{\mathrm{dgp}}v, u=_{\mathrm{dgx}}v, u=_{\mathrm{dgg}}v  \text{ and } u\leq_{\operatorname{Dlex}}v.
\end{array}
\right.	$$

\begin{lem}
	The order   $\leq_{\mathrm{dl}}$ is a well order on $\frakS(X)$.
\end{lem}

\begin{proof}
	Since $\leq_{\mathrm{dl}}$ is a linear order, we only need to	verify that $\leq_{\mathrm{dl}}$ satisfies the descending chain condition. Let $v_1\geq_{\mathrm{dl}} v_2\geq_{\mathrm{dl}} v_3\geq_{\mathrm{dl}}\cdots \in\frakS(X)$. Since the pre-linear order $\leq_{\mathrm{dgp }}$, $\leq_{\mathrm{dgx}}$  and $\leq_{\mathrm{dgg}}$ satisfy the  descending chain condition, there exist $N\geq 1$ and $k\geq 0$ such that $$\deg_P(v_N)=\deg_P(v_{N+1})=\deg_P(v_{N+2})=\cdots=k,$$ $$\deg_X(v_N)=\deg_X(v_{N+1})=\deg_X(v_{N+2})= \cdots$$ and $$\deg_G(v_N)=\deg_G(v_{N+1})=\deg_G(v_{N+2})=\cdots.$$
	Thus all $v_i$ with $i\geq N$  belong to $\frakS_k(X)$. The restriction of the order $\leq_{\mathrm{Dlex}}$ to $\frakS_k(X)$ equals
	to the well order $\leq_{\operatorname{Dlex}_k}$, which by defintion satisfies the descending chain condition, so the chain $v_1\geq_{\mathrm{dl}} v_2\geq_{\mathrm{dl}} v_3\geq_{\mathrm{dl}}\cdots$ stabilizes after finite steps.
\end{proof}

\begin{defn}[{\cite[Definition~5.6]{GGSZ14}}]
	A preorder   $\leq_{\alpha}$ on $\frakM(X)$ (resp. $\frakS(X)$) is called bracket compatible (resp. left compatible, right compatible) if
	$$u \leq_{\alpha} v \Rightarrow\lfloor u\rfloor \leq_{\alpha}\lfloor v\rfloor, \text { (resp. } w u \leq_{\alpha} w v, \text { resp. } u w \leq_{\alpha} v w) ,\  \text { for all }  w \in \frakM(X)\   (resp.\  \frakS(X)) $$
\end{defn}

\begin{lem}[{\cite[Lemma~5.7]{GGSZ14}}]
	A well order $\leq$ is a monomial order on $\frakM(X)$ (resp. $\frakS(X)$) if and only if $ \leq$ is bracket compatible, left compatible and right compatible.
\end{lem}

\begin{thm}\label{Thm:The well order dl is a monomial order}
	The well order $\leq_{\mathrm{dl}}$  is a monomial order on $\frakS(X)$.
\end{thm}

\begin{proof}
	%We will show the order $\leq_{\mathrm{dl}}$ is bracket compatible, left compatible and right compatible.

Let $u\leq_{\dl}v$.
	It is obvious that preorders $\leq_{\mathrm{dgp}}$ and $\leq_{\mathrm{dgx}}$ are bracket compatible, left compatible and right compatible. This solves the case  $u<_{\dgp} v$  and  that of $u=_{\dgp} v$ and $u<_{\dgx} v$.
If $u=_{\dgp} v, u=_{\mathrm{dgx}} v$  and   $u <_{\mathrm{dgg}}v$,
for $w\in \frakS(X)$,  obviously    $\lfloor u\rfloor <_{\mathrm{dgg}}\lfloor v\rfloor$,  $uw <_{\mathrm{dgg}}vw$ and   $wu <_{\mathrm{dgg}}wv$. 	
	
	Now we only need to consider the case that  $u =_{\mathrm{dgp}} v$, $u =_{\mathrm{dgx}} v$, $u =_{\mathrm{dgg}} v$ and $u\leq_{ \mathrm{Dlex}}v$.	
	Let $\deg_P(u) =\deg_P(v)=n$. Since $u,v\in \frakS_n(X)$, thus  $u\leq_{\operatorname{Dlex}_n} v$. By the fact that  the restriction of $\leq_{\operatorname{Dlex}_{n+1}} $ to $\lfloor\frakS_n(X)\rfloor$ is  equal  to $\lfloor\leq_{\operatorname{Dlex}_n}\rfloor$, we have $\lfloor u\rfloor \leq_{\operatorname{Dlex}_{n+1}} \lfloor v\rfloor$ and $\lfloor u\rfloor \leq_{\dl} \lfloor v\rfloor$. Let $w\in\frakS_m(X)$.  One can obtain   $uw\leq_{\operatorname{Dlex}_r}vw$ and $wu\leq_{\operatorname{Dlex}_r}wv$ for $r=\max \left\lbrace m, n \right\rbrace $, so $uw\leq_{\dl} vw$ and $wu\leq_{\dl}wv$.

We are done.
\end{proof}

%Since $\frakS(X)$ is closed under bracket and multiplication, we have the following corollary.

%\begin{coro}
%	The order   $\leq_{\mathrm{dl}}$ is a monomial order on $\frakS(X)$.
%\end{coro}

\begin{remark}\label{remark: dl'}
	In fact, we can define another monomial order $\leq_{\mathrm{dl'}}$ instead of $\leq_{\mathrm{dl}}$ on $\frakS(Z)$, which keeps $\lfloor u\rfloor v>\lfloor uv\rfloor $   for any $u,v\in\frakS(Z)$. To this end, we only need to modify  the definition of G-degree in $\leq_{\mathrm{dl}}$ as follows:
	
	We modify  the definition of G-degree in $\leq_{\mathrm{dl}}$ as follows:	
	for any $u\in\frakS(Z)$, let $u_1,\dots ,u_n\in Z$ be  all the elements occuring in $u$ from left to right. If the half bracket $\lfloor$  (resp. $\rfloor$) is between $u_i$ and $u_{i+1}$ ($1\leq i<n$), define the G-degree of this half bracket to be $n-i$; if there is a  half bracket $\lfloor$ (resp. $\rfloor$)   appearing on the left of $u_1$ (resp. on the right of $u_n$), we define the G-degree of this half bracket to be $n$ (resp. $0$). Denote $\deg_G(u)$ by the sum of the G-degree of all the half brackets in $u$.
\end{remark}	

Now we extend $\leq_{\dl}$ from $\frakS(X)$ to $\frakM(X)$.
\begin{defn}
	Let $X$ be a set with a well order. Let  $\dagger$ be a symbol which is understood to be  $\lfloor 1 \rfloor$ and write $X'=X\sqcup \{\dagger\}$.  Consider the free operated semigroup $\frakS(X')$  over the set  $X'$. The well order on $X$ extends to a  well order $\leq$ on $X'$ by setting $z>\dagger$, for any $z\in X$.
	Then the monomial order $\leq_{\mathrm{dl}}$ on $\frakS(X')$ induces a well order $\leq_{\mathrm{udl}}$ on $\frakM(X)=\frakS(X')\sqcup \{1\}$ (in which $\lfloor 1 \rfloor$ is identified with $\dagger$), by setting $u>1$ for any $u\in \frakS(X')$.
\end{defn}

\begin{remark} Notice that by  writing $\dagger$ instead of $\lfloor 1 \rfloor$, we avoid counting the brackets  of $\lfloor 1 \rfloor$ when computing  $P$-degrees and $G$-degrees.
 \end{remark}

\begin{thm}\label{Thm:The well order udl is a monomial order}
	The well order $\leq_{\mathrm{udl}}$  is a monomial order on $\frakM(X)$ and $\leq_{\mathrm{udl}}$  restricts on $\frakS(X)$ is the momomial order $\leq_{\mathrm{dl}}$.
\end{thm}

\begin{proof}
	For any $x\in \frakM(X)\backslash \{1\}$, $x>_{\mathrm{udl}}1$. We have $\lfloor x\rfloor>_{\mathrm{udl}}x\geq_{\mathrm{udl}}\dagger$. Thus $\leq_{\mathrm{udl}}$ is bracket compatible. Clearly,  $\leq_{\mathrm{udl}}$ is left and right compatible.
\end{proof}

%Note that the orders $\leq_{\mathrm{dl}}$ and $\leq_{\mathrm{udl}}$ on $\frakM(X)$ are different, for example, we have $\lfloor x \rfloor y \leq_{\mathrm{dl}}x\lfloor 1 \rfloor y$ by $G$-degree, whereas the following lemma shows it is different for $\leq_{\mathrm{udl}}$.
We record an important  technical result.
\begin{prop}\label{Lemma: inequal  relation}
	For any $ u,v\in \frakM(X)\backslash \{1\}$, we have
	\begin{itemize}
		\item [(a)] $\lfloor u \rfloor v  <_{\mathrm{udl}} \lfloor u v\rfloor  <_{\mathrm{udl}} u\lfloor v\rfloor<_{\mathrm{udl}} \lfloor u\rfloor\lfloor v\rfloor;$
		\item [(b)] $u\lfloor 1 \rfloor v <_{\mathrm{udl}}\lfloor u \rfloor v;$
		\item [(c)] $\lfloor1\rfloor  u v <_{\mathrm{udl}}\lfloor u \rfloor v;$
		\item [(d)] $   u   <_{\mathrm{udl}}\lfloor 1 \rfloor u;$
        \item [(e)] $   u   <_{\mathrm{udl}} u \lfloor 1 \rfloor.$
	\end{itemize}
%	\begin{itemize}
%		\item [(a)] $u\lfloor v\rfloor<_{\mathrm{udl}} \lfloor u\rfloor\lfloor v\rfloor $;
%		\item [(b)] $\lfloor u v\rfloor  <_{\mathrm{udl}} u\lfloor v\rfloor$;
%		\item [(c)] $\lfloor u \rfloor v  <_{\mathrm{udl}} u\lfloor v\rfloor$;
%		\item [(d)] $u\lfloor 1 \rfloor v <_{\mathrm{udl}} u\lfloor v\rfloor$;
%		\item [(e)] $\lfloor1\rfloor  u v <_{\mathrm{udl}} u\lfloor v\rfloor$.
%	\end{itemize}

\end{prop}

\begin{proof}
	Let  $ u,v\in  \frakM(X)\backslash \{1\}=\frakS(X')$.
	\begin{itemize}
		\item [(a)]
 $\lfloor u \rfloor v  <_{\mathrm{udl}} \lfloor u v\rfloor $ holds because     $$\deg_G(\lfloor uv\rfloor)-\deg_G(\lfloor u\rfloor v)=\deg_{X'}(v)>0,$$
 $u\lfloor v\rfloor<_{\mathrm{udl}} \lfloor u\rfloor\lfloor v\rfloor$ follows  from $u\lfloor v\rfloor<_{\mathrm{dgp}} \lfloor u\rfloor\lfloor v\rfloor $ and
		$\lfloor u v\rfloor  <_{\mathrm{udl}} u\lfloor v\rfloor$  is obtained from  $\deg_G(u\lfloor v\rfloor)-\deg_G(\lfloor uv\rfloor)=\deg_{X'}(u)>0$.
		
		\item [(b)]  The inequality is valid as  $u\lfloor 1 \rfloor v=u\dagger v <_{\mathrm{dgp}} u\lfloor v\rfloor$.
		\item [(c)]  The statement can be deduced  from $\lfloor1\rfloor  u v =\dagger uv<_{\mathrm{dgp}} u\lfloor v\rfloor$.
		\item [(d)]  The assertion follows from   $   \deg_{X'} (u)   < \deg_{X'} (\dagger u)=\deg_{X'} (\lfloor 1 \rfloor u)$.
		\item [(e)]  We infer the result from  $  \deg_{X'} (u)   <\deg_{X'} (u\dagger)= \deg_{X'}(u \lfloor 1 \rfloor) $.
	\end{itemize}
\end{proof}

Next, we recall an another monomial order $\leq_{\db}$ on $\frakM(X)$ constructed by Gao, Guo, Sit and Zheng \cite{GGSZ14}.

\begin{itemize}
	\item [(a)] Let $u,v\in \frakM_0(X)$, define
	$$u\leq_{\db_0} v\Leftrightarrow u\leq_{\mathrm{dlex}}v.$$
	\item[(b)]  Suppose that for each $1\leq i\leq n$, we are given a  well order $\leq_{\db_i} $ on $\frakM_i(X)$ such that $\leq_{\db_i}$ extends $\leq_{\db_{i-1}}$.
	For $r\geq0$, denote  $\frakM_{n+1, r}(X)=\left\lbrace u\in \frakM_{n+1}(X)~|~ |u|_P=r \right\rbrace $.
 Let $u,v\in \frakM_{n+1, r}(X)$, so they can be written as $$u=u_{0}\lfloor u_{1}^{*}\rfloor u_{1}\lfloor u_{2}^{*}\rfloor \cdots\lfloor u_{r}^{*} \rfloor u_{r}, v=v_{0}\lfloor v_{1}^{*}\rfloor v_{1}\lfloor v_{2}^{*}\rfloor \cdots\lfloor v_{r}^{*} \rfloor v_{r},  $$ where
  $u_{0}, u_{1}, \dots, u_{r},$
	$v_{0}, v_{1}, \dots, v_{r} \in \calm(X)$ and $u_{1}^{*}, u_{2}^{*}, \dots, u_{r}^{*}, v_{1}^{*}, v_{2}^{*}, \dots, v_{r}^{*} \in \frakM_n(X) .$ Then  impose
		$$ u\leq_{\mathrm{lex}_{n+1}}v\Leftrightarrow (u_{1}^{*}, u_{2}^{*}, \dots, u_{r}^{*},u_{0}, u_{1}, \dots, u_{r})\leq_{\mathrm{clex}}( v_{1}^{*}, v_{2}^{*}, \dots, v_{r}^{*},v_{0}, v_{1}, \dots, v_{r}),$$
where $\leq_{\mathrm{clex}}$ is the lexicographical order on $\frakM_n(X)^r\times \calm(X)^{r+1}$ induced by $\leq_{\db_n}$.
	Now the  well order $\leq_{\mathrm{db}_{n+1}}$ on $\frakM_{n+1}(X)$ is given by
		$$u\leq_{\mathrm{db}_{n+1}}v\Leftrightarrow u\leq_{\mathrm{dgp,brp},\mathrm{lex}_{n+1}}v\Leftrightarrow \left\{
		\begin{array}{lcl}
			u<_{\mathrm{dgp}} v, \text{or }\\
			u=_{\mathrm{dgp}}v \text{ and }u<_{\mathrm{ brp}} v, \text{or }\\
			u=_{\mathrm{dgp}}v, u=_{\mathrm{brp}}v   \text{ and } u<_{\mathrm{lex}_{n+1}} v.
		\end{array}
		\right.	$$
\end{itemize}
By a limit process, one can get a monomial order $\leq_{\db}$ on $\frakM(X)$; for details, see  \cite[Theorem~5.8]{GGSZ14}.

%$\leq_{\dl'}$ $\leq_{\db}$

\section{Operator  polynomial identities and operated  GS bases}
\label{Section: Operator  polynomial identities and oeprator GS bases}

In this section, we recall basic facts about free operated algebras whose operator satisfies some  polynomial identities and the GS basis theory in this setup.
We will also recall differential type algebras which are the main objects in this paper.

\subsection{Free operated algebras and operator  polynomial identities}\

%Free objects in operated contexts have been constructed in \cite{QQWZ}.

In this subsection, we   recall   some basic notions and facts   related to   free operated  algebras and operator  polynomial identities.

Let $X$ be a set.

 % One can construct a family of monoid homomorphisms:
%$$	j_{n-1,n}: \frakM_{n-1}(X)= \calm(X\sqcup  \left \lfloor \frakM_{n-2}(X)\right \rfloor) \hookrightarrow  \frakM_{n}(X) = \calm(X\sqcup  \left \lfloor \frakM_{n-1}(X)\right \rfloor) \ \ (n\geq 2)$$
%induced from the natural  injection
%$$ j_{0,1}: \frakM_0(T):=T \hookrightarrow \frakM_1(X):= \calm(X\sqcup  \left \lfloor X\right \rfloor)
%$$
%and
%$$
%\Id_X\sqcup \left \lfloor j_{n-2,n-1}\right \rfloor :X\sqcup\left \lfloor \frakM_{n-2}(X)\right \rfloor\hookrightarrow X\sqcup\left \lfloor \frakM_{n-1}(X)\right \rfloor \ \ (n\geq 2).
%$$
%It is show in \cite{Guo09a} that $  \frakM(X)=\varinjlim \frakM_{n}(X) $ is the free operated monoid over $X$.

\begin{defn}[{\cite{GSZ13, GGSZ14,    GaoGuo17}}]
	   We call an element     $\phi(x_1,\dots,x_n)   \in \bk\frakS(X) $ (resp. $\bk\frakM(X)$) with $ n\geq 1, x_1, \dots, x_n\in X$ an operated polynomial identity (aka  OPI).

\textit{{From now on, we always assume that OPIs are multilinear, that is,  they are linear in each $x_i$.}}

\end{defn}
%Let $\phi=\phi(x_1,\dots,x_n)$ be an OPI  in $\bk\frakS(X)$ (resp. $\bk\frakM(X)$) and  $(A,P)$ be   an operated algebra (resp. unital operated algebra). Let $\theta: X\rightarrow A$ be a map with $r_i= \theta(x_i), 1\leq i\leq n$.
%By the universal property of $\bk\frakS(X) $ (resp. $\bk\frakM(X)$), there is a unique morphism of (unital) operated algebras $\tilde{\theta}: \bk\frakS(X)\rightarrow A$ (resp. $\tilde{\theta}: \bk\frakM(X)\rightarrow A$).  We denote:
%\[ \phi(r_1,\dots,r_n):=\tilde{\theta}(\phi(x_1,\dots,x_n)). \]

\begin{defn}[{\cite{GSZ13, GGSZ14,   GaoGuo17}}]  Let $\phi(x_1,\dots,x_n)$ be an   OPI. A  (unital) operated algebra $(A,P)$ is said to satisfy the OPI $\phi(x_1,\dots,x_n)$ if
	$\phi(r_1,\dots,r_n)=0,$ for all $r_1,\dots,r_n\in A.$
In this case, $(A,P)$ is called a (unital) $\phi$-algebra and $P$ is called a $\phi$-operator.
	
	Generally, for a family of OPIs $\Phi$, we call a  (unital)  operated algebra    $(A,P)$ a  (unital)  $\Phi$-algebra    if it is a  (unital) $\phi$-algebra for any $\phi\in \Phi$.
	Denote the  category of $\Phi$-algebras (resp. unital $\Phi$-algebras)  by $\Phi\zhx\Alg$ (resp. $\Phi\zhx\uAlg$).

%If $\Phi=\{\phi\}$, write
%$\Phi\zhx\Alg=\phi\zhx\Alg$.
\end{defn}

\begin{defn}[{\cite{GSZ13, GGSZ14,    GaoGuo17}}]   An operated ideal  of an operated algebra  is   an  ideal   of the associative algebra  closed under the action of the operator.
	The operated ideal generated by a subset $S$ is denoted by $\left\langle S\right\rangle_\OpAlg$ (resp. $\left\langle S\right\rangle_\uOpAlg$).
\end{defn}
Obviously   the quotient of an operated algebra (resp. unital operated algebra) by an operated  ideal is naturally an operated algebra (resp. unital operated algebra).

	From now on, $\Phi$ denotes  a family  of OPIs    in $\bk\frakS(X) $ or   $\bk\frakM(X)$.  For a set $Z$ and  a subset  $Y$    of $ \frakM(Z)$, introduce the subset $S_\Phi(Y)\subseteq\bk \frakM(Z)$ to be
\[S_{\Phi}(Y):=\{\phi(u_1,\dots,u_k)\ |\ u_1,\dots,u_k\in Y,~\phi(x_1,\dots,x_k)\in \Phi\}.\]

We will consider free (unital)  $\Phi$-algebra over an associative algebra which is defined via the usual universal property; see for example
\cite[Proposition 4.7]{QQWZ}.
The following result gives a construction of free (unital)  $\Phi$-algebra over an associative algebra.
\begin{prop}[{\cite[Proposition 4.8]{QQWZ}}]\label{Prop: from algebra to Phi algebra with relations}
	\begin{itemize}
		\item [(a)] Let $\Phi\subset\bk\frakS(X)$ and $A=\bk \cals (Z) \slash I_A$ an algebra. Then
		$$ \calf_{\Alg}^{\Phi\zhx\Alg}(A):=\bk\frakS(Z)\slash\left\langle  S_{\Phi}(\frakS(Z))\cup I_A\right\rangle_\OpAlg$$
		is the free $\Phi$-algebra generated by  $A$.
		\item [(b)] Let $\Phi\subset\bk\frakM(X)$ and $A=\bk \calm (Z) \slash I_A$ a unital algebra. Then
		 $$\calf_{\uAlg}^{\Phi\zhx\uAlg}(A):=\bk\frakM(Z)\slash\left\langle  S_{\Phi}(\frakM(Z))\cup I_A\right\rangle_\uOpAlg$$
		 is the free unital $\Phi$-algebra over $A$.
	\end{itemize}
\end{prop}

We introduce a notion of equivalences between families of OPIs.
\begin{defn}\label{Def: equivalent OPIs}
	We say two families of OPIs $\Phi_1,\Phi_2\subset \bk\frakS(X)$  are equivalent in $\bk\frakS(X)$ if  	$$\left\langle S_{\Phi_1}(\frakS(Z))\right\rangle_\OpAlg=\left\langle S_{\Phi_2}(\frakS(Z))\right\rangle_\OpAlg$$ for any set $Z$; we say two families of OPIs $\Phi_1,\Phi_2\subset \bk\frakM(X)$  are equivalent in $\bk\frakM(X)$ if $$\left\langle S_{\Phi_1}(\frakM(Z))\right\rangle_\uOpAlg=\left\langle S_{\Phi_2}(\frakM(Z))\right\rangle_\uOpAlg$$ for any set $Z$.
\end{defn}

\subsection{Operated GS  bases for free $\Phi$-algebras}\

\begin{defn} [\cite{BokutChenQiu, GSZ13, GGSZ14, GaoGuo17}]
	Let $Z$ be a set,  $\leq$ a linear order on $\frakM(Z)$ and $f \in \bk \frakM(Z)$.
	\begin{itemize}
		\item [(a)] Let $f\notin \bk$. The leading monomial of $f$ , denoted by $\bar{f}$, is the largest monomial appearing in $f$. The leading coefficient of $f$ , denoted by $c_f$, is the coefficient of $\bar{f}$ in $f$. We call $f$	monic with respect to $\leq$ if $c_f = 1$.
		\item [(b)] Let $f\in \bk$ (including the case $f=0$).  We define the leading monomial of $f$ to be $1$ and the	leading coefficient of $f$ to be $c_f=f$.
		\item [(c)] A subset $S\subseteq \bk\frakM(Z)$ is called monicized with respect to $\leq$,  if each nonzero element of $S$ has leading coefficient $1$. Obviously, each subset $S\subseteq \frakM(Z)$ can be made monicized  if we divide  each nonzero  element by    its  leading coefficient.
	\end{itemize}
\end{defn}

We need  another notation.
Let $Z$ be a set. For  $u\in\frakM(Z)$ with $u\neq1$, as  $u$ can be uniquely written as a product $ u_1 \cdots u_n $ with $ u_i \in Z \cup \left \lfloor\frakM(Z)\right \rfloor$ for $1\leq i\leq n$, call $n$ the breadth of $u$, denoted by $|u|$; for $u=1$, we define
$ |u| = 0 $.

\begin{defn} [\cite{BokutChenQiu, GSZ13, GGSZ14, GaoGuo17}]
	Let $\leq$ be a monomial order on $\frakS(Z)$ (resp. $\frakM(Z)$)  and $ f, g \in \bk\frakS(Z)$ (resp. $\bk\frakM(Z)$) be monic.
	\begin{itemize}
		\item[(a)]If there are $w,u,v\in \frakS(Z)$ (resp. $\frakM(Z)$) such that $w=\bar{f}u=v\bar{g}$ with max$\left\lbrace |\bar{f}| ,|\bar{g}|   \right\rbrace < \left| w\right| < |\bar{f}| +|\bar{g}|$, we call
		$$\left( f,g\right) ^{u,v}_w:=fu-vg$$
		the intersection composition of $f$ and $g$ with respect to $w$.
		\item[(b)] If there are $w\in\frakS(Z)$ (resp. $\frakM(Z)$) and $q\in\frakS^\star(Z)$ (resp. $\frakM^\star(Z)$) such that $w=\bar{f}=q|_{\bar{g}}$, we call
		$$\left( f,g\right)^q_w:=f-q|_g $$
		the inclusion composition of $f$ and $g$ with respect to $w$.
		
	\end{itemize}
\end{defn}
\begin{defn} [\cite{BokutChenQiu, GSZ13, GGSZ14, GaoGuo17}]
	Let $Z$ be a set and $\leq$ a monomial order on $\frakS(Z)$ (resp. $\frakM(Z)$). Let $\calg\subseteq \bk\frakS(Z) $ (resp. $\bk\frakM(Z)$).
	\begin{itemize}
		\item[(a)]An element $f\in \bk\frakS(Z) $ (resp. $\bk\frakM(Z) $) is called trivial modulo $(\calg,w)$ for $w\in \frakS(Z)$ ($\frakM(Z)$) if
		$$f=\underset{i}{\sum}c_iq_i|_{s_i}\text{ with }q_i|_{\bar s_i}<w\text{, where }c_i\in \bk,\ q_i\in\frakS^\star(Z)~(\text{resp.}~\frakM^\star(Z))\  \mathrm{and}\  s_i\in \calg.$$
		\item[(b)]  The subset $\calg$ is called a  GS basis in $\bk\frakS(Z)$ (resp. $\bk\frakM(Z)$) with respect to $\leq$ if, for all pairs $f,g\in \calg$ monicized with respect to $\leq$, every  intersection composition of the form $\left( f,g\right) ^{u,v}_w$ is trivial modulo $(\calg,w)$,  and every inclusion composition of the form $\left( f,g\right)^q_w$ is trivial modulo $(\calg,w)$.
%		\item[(c)]  Notice that the monomial order $\leq$ restrict on $\bk\frakS(Z) $ is   a monomial order $\leq'$  on $\frakS(Z)$. The subset $\calg\subseteq \bk\frakS(Z) $ is called a  GS basis in $\bk\frakS(Z) $ with respect to $\leq'$ if,		 $\calg$ is a  GS basis in $\bk\frakM(Z) $ with respect to $\leq$.
	\end{itemize}
\end{defn}

\textit{ {To distinguish from usual GS bases for associative algebras,  from now on, we shall rename       GS bases  in operated contexts by operated GS bases.}}

\begin{prop}\label{GSB for operated algebras considered as GSB for unital operated algebras}
	Let $Z$ be a set and $\leq$ a monomial order on $\frakM(X)$. Then its restriction to $\frakS(X)$ is also a monomial order.  Moreover, $\calg \subseteq \bk \frakS(Z) $    is an operated GS basis in $\bk\frakS(Z)$  with respect to the restriction of $\leq$  is equivalent to that $\calg$ is an operated GS basis in  $\bk\frakM(Z) $ with respect to $\leq$.
\end{prop}

\begin{thm} [\cite{BokutChenQiu, GSZ13, GGSZ14, GaoGuo17}]\label{Thm: unital CD}
	(Composition-Diamond Lemma) Let $Z$ be a set, $\leq$ a monomial order on $\frakM(Z)$ and $\calg\subseteq \bk\frakM(Z) $. Then the following conditions are equivalent:
	\begin{itemize}
		\item[(a)] $\calg$ is an operated GS basis in $\bk\frakM(Z) $.
		\item[(b)]  Denote
		$$\Irr(\calg):=\frakM(Z)\backslash \left\lbrace q|_{\overline{s}}~|~s\in \calg,\ \ q\in\frakM^\star(Z)\right\rbrace. $$
		As a $\bk$-space, $\bk\frakM(Z) =\bk\Irr(\calg)\oplus\left\langle \calg \right\rangle_\uOpAlg$ and     $ \Irr(\calg) $ is a $\bk$-basis of $\bk\frakM(Z)\slash\left\langle\calg \right\rangle_\uOpAlg$.
	\end{itemize}
\end{thm}

\begin{thm}[\cite{BokutChenQiu, GSZ13, GGSZ14, GaoGuo17}]\label{Thm: nonunital CD}
	(Composition-Diamond Lemma) Let $Z$ be a set, $\leq$ a monomial order on $\frakS(Z)$ and $\calg\subseteq \bk\frakS(Z) $. Then the following conditions are equivalent:
	\begin{itemize}
		\item[(a)] $\calg$ is an operated GS basis in $\bk\frakS(Z) $.
		\item[(b)]   Denote
		$$\Irr(\calg):=\frakS(Z)\backslash \left\lbrace q|_{\overline{s}}~|~s\in \calg,\ \ q\in\frakS^\star(Z)\right\rbrace. $$
		As a $\bk$-space, $\bk\frakS(Z) =\bk\Irr(\calg)\oplus\left\langle \calg \right\rangle_\OpAlg$ and     $ \Irr(\calg)$ is a $\bk$-basis of $\bk\frakS(Z)\slash\left\langle\calg \right\rangle_\OpAlg$.
	\end{itemize}
\end{thm}

The following result refines slightly  \cite[Proposition 6.7]{QQWZ}.
\begin{prop} \label{GS-basis for nonunital algebras considered as GS-basis for nonunital operated algebras}
	Let $Z$ be a set and $\leq$ a monomial order on $\frakS(Z)$ (resp. $\frakM(Z)$). Clearly when  restricted to $ \cals(Z)$ (resp. $\calm(Z)$), it is still a monomial order.
	Then a GS basis $G \subseteq \bk \cals(Z) $  (resp. $\bk\frakM(Z)$)   with respect to the restriction of $\leq$ is also an operated GS basis in  $\bk\frakS(Z) $ (resp. $\bk\frakM(Z)$) with respect to $\leq$.
\end{prop}

%Now, let's consider operated  GS  bases for free $\Phi$-algebras over  unital algebras.
\begin{defn}[\cite{GaoGuo17}]
	\begin{itemize}
		\item [(a)]	Let  $\Phi\subseteq \bk\frakS(X)$ be  a  family   of OPIs. Let $Z$ be a set and $\leq$ a monomial
		order on $\frakS(Z)$. We call $\Phi$  GS  on $\bk\frakS(Z)$  with respect to $\leq$ if $S_{\Phi}(\frakS(Z))$    is an operated
		GS  basis  in $\bk\frakS(Z)$   with respect to $\leq$.
		\item [(b)]Let   $\Phi\subseteq  \bk\frakM(X)$ be  a  family  of OPIs. Let $Z$ be a set and $\leq$ a monomial
		order on  $\frakM(Z)$. We call $\Phi$   GS  on  $\bk\frakM(Z)$  with respect to $\leq$ if  $S_{\Phi}(\frakM(Z))$    is an operated
		GS basis  in   $\bk\frakM(Z)$  with respect to $\leq$.
	\end{itemize}

	%{		\begin{itemize}	\item [(a)] \item[(b)] We call $\Phi'$ is OPIs GS  of $\Phi$ with respect to $\leq$ if there is a superset $\Phi'\subseteq \bk\frakM(X)$ of $\Phi$ such that $\left\langle R_{\Phi}(Z) \right\rangle_\uOpAlg=\left\langle R_{\Phi'}(Z) \right\rangle_\uOpAlg$ and $\Phi'$ is a GS  in $\bk\frakM(Z)$ with respect to $\leq$.	\end{itemize}}
\end{defn}

For a given unital algebra $A$ and a  GS   family   of  OPIs $\Phi$ under some monomial order, one can give an operated GS basis of the free unital $\Phi\zhx$algebra over $A$ by \cite[Theorem~5.9]{QQWZ}.

\begin{thm}[{\cite[Theorem 5.9]{QQWZ}}]\label{Thm: GS basis for free unital Phi algebra over unital alg}
	Let $X$ be a set and $\Phi\subseteq \bk\frakM(X)$ a system of OPIs. Let $A=\bk \calm (Z) \slash I_A$ be a unital algebra with generating set $Z$.
	Assume that $\Phi$ is operated GS  on $Z$ with respect to a monomial order $\leq$ in $\frakM(Z)$ and that  $G$ is a GS    basis of $I_{A}$ in $\bk \calm (Z)$ with respect to the restriction of $\leq$ to $ \calm(Z)$.

	Suppose that the leading monomial  of any OPI $\phi(x_1, \dots, x_n)\in \Phi$ has no subword in $\calm(X)\backslash X$,  and  that  for all $u_1, \dots, u_n\in \frakM (Z)$,  $\phi(u_1, \dots, u_n) $   vanishes or its  leading monomial is still    $\overline{\phi}(u_1, \dots, u_n) $.  Then $S_{\Phi}(\frakM(Z))\cup G$ is an operated    GS basis of $\left\langle S_{\Phi}(\frakM(Z))\cup I_A\right\rangle_\uOpAlg$ in $\bk\frakM(Z)$ with respect to $\leq$.
\end{thm}

%
%By Theorem~\ref{Thm: unital CD}, we have the following result.
%\begin{coro}[{\cite[Corollary 5.10]{QQWZ}}]\label{Cor: linear basis for free Phi algebra over unital alg}
%	With the same assumption as Theorem~\ref{Thm: GS basis for free unital Phi algebra over unital alg}, denote by $$\eta:\bk\frakM(Z)\rightarrow{\calf_{\uAlg}^{\Phi\zhx\Alg}(A)=\bk\frakM(Z)\slash\left\langle { S_{\Phi}(\frakM(Z))}\cup I_A\right\rangle_\uOpAlg}$$
%	the natural quotient  map.
%	% $$\Irr(S_{\Phi}(\frakM(Z))\cup G)=\calf_{\OpMon}^{\Set}(Z)\backslash \left\lbrace q|_{\overline{s}}~|~s\in R_{\Phi'}(Z)\cup G,q\in\frakM^\star(Z)\right\rbrace,$$
%	Then $\eta(\Irr(S_{\Phi}(\frakM(Z))\cup G))$  is a $\bk$-basis of ${\calf_{\uAlg}^{\Phi\zhx\Alg}(A)}$.
%\end{coro}

We also have a nonunital version of this result, whose proof is similar to that of Theorem~\ref{Thm: GS basis for free unital Phi algebra over unital alg}, so we omit it.  %Although the proof is similar, we still repeat the proof in the following.

\begin{thm}\label{Thm: GS basis for free nonunital Phi algebra over nonunital alg}
	Let $X$ be a set and $\Phi\subseteq \bk\frakS(X)$ a system of OPIs. Let $A=\bk \cals (Z) \slash I_A$ be an algebra with generating set $Z$.
	Assume that $\Phi$ is operated GS  on $Z$ with respect to a monomial order $\leq$ in $\frakS(Z)$ and that  $G$ is a GS    basis of $I_{A}$ in $\bk \cals(Z)$ with respect to the restriction of $\leq$ to $ \cals(Z)$.
	
	Suppose that the leading monomial  of any OPI $\phi(x_1, \dots, x_n)\in \Phi$ has no subword in $\cals(X)\backslash X$,  and  that  for all $u_1, \dots, u_n\in \frakS (Z)$,  $\phi(u_1, \dots, u_n) $   vanishes or its  leading monomial is still    $\overline{\phi}(u_1, \dots, u_n) $.  Then $S_{\Phi}(\frakS(Z))\cup G$ is an operated  GS basis of $\left\langle S_{\Phi}(\frakS(Z))\cup I_A\right\rangle_\OpAlg$ in $\bk\frakS(Z)$ with respect to $\leq$.

\end{thm}

\subsection{Remainder on  rewriting systems}\

We need some basic notions  about rewriting systems in order to introduce OPIs of  differential type. The basic references about rewriting systems are, for instance,  \cite{BN} and the recent lecture notes \cite{Malbos}.
\begin{defn}
	Let $V$ be a $\bf{k}$-space with a $\bf{k}$-basis $Z$.
	\begin{itemize}
		\item[(a)] For $f=\sum_{w \in Z} c_{w} w \in V$ with $c_{w} \in \mathbf{k},$ the support $\operatorname{Supp}(f)$ of $f$ is the set $\{w \in Z \mid c_{w} \neq 0\} .$ By convention, we take $\operatorname{Supp}(0)=\emptyset$.
		\item[(b)]  Let $f, g \in V$. We use $f \dot{+} g$ to indicate the property that $\operatorname{Supp}(f) \cap \operatorname{Supp}(g)=\emptyset .$ If this is the case, we say $f+g$ is a direct sum of $f$ and $g,$ and use $f \dot{+} g$ also for the sum $f+g$.
		\item[(c)] For $f \in V$ and $w \in \operatorname{Supp}(f)$ with the coefficient $c_{w}$, write $R_{w}(f):=c_{w} w-f \in V$. So $f=c_{w} w-R_{w}(f)$.
	\end{itemize}
\end{defn}

\begin{defn}
	Let $V$ be a $\bf{k}$-space with a $\bf{k}$-basis $Z$.
	\begin{itemize}
		\item[(a)] A term-rewriting system $\Pi$ on $V$ with respect to $Z$ is a binary relation $\Pi \subseteq Z \times V$. An element $(t, v) \in \Pi$ is called a (term-)rewriting rule of $\Pi,$ denoted by $t \rightarrow v$.
		\item[(b)] The term-rewriting system $\Pi$ is called simple with respect to $Z$ if $t \dot{+} v$ for all $t \rightarrow v \in \Pi$
		\item[(c)]  If $f=c_{t} t-R_{t}(f) \in V,$ using the rewriting rule $t \rightarrow v,$ we get a new element $g:=c_{t} v-R_{t}(f) \in V,$ called a one-step rewriting of $f$ and denoted by $f \rightarrow_{\Pi} g$.
		\item[(d)]  The reflexive-transitive closure of $\rightarrow_{\Pi}$ (as a binary relation on $V$ ) is denoted by $\stackrel{\ast}{\rightarrow}_{\Pi}$ and, if $f \stackrel{\ast}{\rightarrow}_\Pi g,$ we say $f$ rewrites to $g$ with respect to $\Pi$.
		\item[(e)] Two elements $f,g\in V$ are joinable if there exists $h\in V$ such that $f \stackrel{\ast}{\rightarrow}_{\Pi}h$ and $g \stackrel{\ast}{\rightarrow}_{\Pi}h$; we denote this by $f\downarrow_{\Pi} g$.
		\item[(f)] A term-rewriting system $\Pi$ on $V$ is called terminating if there is no infinite chain of one-step rewriting $$f_{0} \rightarrow_{\Pi} f_{1} \rightarrow_{\Pi} f_{2} \cdots.$$
		\item[(g)] A term-rewriting system $\Pi$ is called compatible with a linear order $\geq$ on $Z$, if $t>\bar v$ for each $t\rightarrow v\in\Pi$.
	\end{itemize}
\end{defn}

The following theorem supplies a method to show that a single OPI $\phi$ is GS by rewriting system. For convenience, we will use the notation
$\mathfrak{u}=(u_1,\dots,u_{k})\in \frakM(Z)^k$ (similar for $\mathfrak{u}\in \frakS(Z)^k$), and $\phi(\mathfrak{u})$ for $\phi(u_1,\dots,u_{k})$ in the following. We say that an element $f\in\bk\frakM(Z)$ is in $\phi$-normal form if no monomial of $f$ contains any subword of the form $\overline{\phi(\mathfrak{u})}$ with $\mathfrak{u}\in \frakM(Z)^k$.

\begin{thm}[{\cite[Theorem 4.1]{GaoGuo17}}]\label{Thm: verify a GS OPI}
	Let $\phi(x_1,\dots,x_{k}) \in \mathbf{k} \mathfrak{M}(X)$ be a multilinear OPI such that  $\phi=\overline{\phi}-R(\phi)$, where $R(\phi)$ is in $\phi$-normal form. Suppose that, for any set $Z,$ there is a monomial order $\leq$ on $\mathfrak{M}(Z),$ such that the following two conditions hold:
	\begin{itemize}
		\item[(a)] if $\phi(\mathfrak{u}), \phi(\mathfrak{v}) \in S_{\phi}({\frakM(Z)})$ are such that $\overline{\phi(\mathfrak{u})}=\alpha\beta$ and $\overline{\phi(\mathfrak{v})}=\beta\gamma$ for some $\alpha, \beta, \gamma\in\mathfrak{M}(Z)$ and $\mathfrak{u}, \mathfrak{v} \in \mathfrak{M}(Z)^{k},$ then
		$R(\phi(\mathfrak{u})) \gamma \downarrow_{\Pi_\phi(Z)} \alpha R(\phi(\mathfrak{v})),$ where the
		term-rewriting system $$\Pi_\phi(Z):=\left\{~q|_{\overline{{\phi}(\mathfrak{u})}}\rightarrow q|_{R(\phi(\mathfrak{u}))}~\big|~\mathfrak{u}\in\frakM(Z)^k,~q\in\frakM^\star(Z) ~\right\}$$ is compatible with $\leq$;
		\item[(b)] if $\overline{\phi(\mathfrak{u})}=q|_{\overline{\phi(\mathfrak{v})}}$ for some $\star \neq q \in \mathfrak{M}^{\star}(Z)$ and $\mathfrak{u}, \mathfrak{v} \in \mathfrak{M}(Z)^{k}$, then $\overline{\phi(\mathfrak{v})}$ is a subword
		of some $u_{i}, 1 \leq i \leq k$.
	\end{itemize}
	Then the OPI $\phi$ is GS with respect to the order $\leq$.
\end{thm}

\begin{lem}\label{Lemma:joinable gives trivial modulo} Assume that  $X=\{x_1,\dots,x_n\}$.
Let $\Phi\subset\bk\frakS(X)$ be a system of OPIs, $Z$ a set  and $\leq$ a monomial order  on $\frakS(Z)$.
Consider the  rewriting system defined as follows
$$\Pi_\Phi(Z):=\left\lbrace ~ q|_{\overline{\phi(\mathfrak{u})}}\rightarrow q|_{\overline{\phi(\mathfrak{u})}-\phi(\mathfrak{u})} ~|~ \phi\in\Phi,~\mathfrak{u}\in\frakS(Z)^k,~q\in\frakS^\star(Z) ~\right\rbrace. $$
If $f,g\in \bk\frakS(Z)$ are joinable, then for any $w\in\frakS(Z)$ such that $ w >\overline f$ and $ w >\overline g$, we have $f-g$ is trivial modulo $(S_{\Phi}(\frakS(Z)),w)$.
\end{lem}

\begin{proof}
	By the assumption, there exists $h\in \frakS(Z)$ such that  $f \stackrel{\ast}{\rightarrow}_{\Pi_\Phi(Z)}h$ and $g \stackrel{\ast}{\rightarrow}_{\Pi_\Phi(Z)}h$. By the definition of rewriting, there exist $l_1,\dots,l_s,l_1',\dots,l_t'\in S_{\Phi}(\frakS(Z))$ and  $q_1,\dots,q_s,q'_1,\dots,q'_t\in\frakS^\star(Z)$ such that $f-h=\sum_{i=1}^s q_i|_{l_i}$ and $g-h=\sum_{j=1}^t q'_j|_{l_j'}$, where $ q_i|_{\bar{l_i}}\leq \overline f,~1\leq i\leq s$ and $ q'_j|_{\bar{l_j'}}\leq \overline g,~1\leq j\leq t$. Thus,
	$$f-g=f-h+h-g=\sum_{i=1}^s q_i|_{l_i}-\sum_{j=1}^t q'_j|_{l_j'}$$
	is trivial modulo $(S_{\Phi}(\frakS(Z)),w)$.
\end{proof}

 \section{Differential type OPIs}\

In this section, we recall differential OPIs due to Guo et al. \cite{GSZ13,GaoGuo17} and we will show that differential type OPIs appearing in their classification are equivalent to some new OPIs to which one can apply Theorem~\ref{Thm: GS basis for free nonunital Phi algebra over nonunital alg} and  Theorem~\ref{Thm: GS basis for free unital Phi algebra over unital alg}.

\medskip

 From now on   $X=\{x,y\}$  is a set with two elements.

\subsection{Unital differential type OPIs}

\begin{defn}[\cite{GSZ13,GaoGuo17}]
	An OPI $\phi\in\bk\frakM(X)$ is said to be of unital differential type if $\phi$ is of the form $\lfloor x y\rfloor-N(x, y)$, where $N(x, y)$ satisfies the following conditions:
	\begin{itemize}
		\item[(a)] $N(x, y)$ is linear in $x$ and $y$;
		%in the sense that the total degree of $\lfloor x\rfloor^{n},~n \geq0$ (resp. $\lfloor y\rfloor^{n}$, $n \geq 0$ ) in each monomial of $N(x, y)$ is one
		\item[(b)] no monomial of $N(x, y)$ contains any subword of the form $\lfloor uv\rfloor$ for any $u,v\in\frakM(\left\lbrace x,y\right\rbrace)\backslash\{1\}$;%N(x,y) is in $\phi(x,y)$-normal form
		\item[(c)] for any set $Z$ and $u, v, w \in\frakM(Z )\backslash\{1\}$, $$N(u v, w)-N(u, v w)\stackrel{\ast}\rightarrow_{\Pi_{\phi}(Z)}0,$$
		where $\Pi_{\phi}(Z):= \left\{q|_{\lfloor uv\rfloor}\rightarrow q|_{N(u,v)}~|~u,v \in \frakM(Z)\backslash\{1\},~q\in\frakM^\star(Z)\right\}$.
	\end{itemize}
	%If $\phi:=\lfloor x y\rfloor-N(x, y)$ is an OPI of differential type, we also say that the expression $N(x, y)$ and the defining operator $P$ of a $\phi$ -algebra $R$ are of differential type.
\end{defn}

It is showed in \cite[Theorem 5.7]{GSZ13} that for an OPI $\phi$ of differential type, $S_{\phi}(\frakM(Z))$ is an operated GS  basis of $\langle S_{\phi}(\frakM(Z))\rangle_\uOpAlg$ with respect to some monomial order, under which the leading monomial of $\phi$ is $\lfloor xy\rfloor$.

Guo, Sit and Zheng  gave a list of OPIs of differential type and   they conjectured that these are all possible OPIs of differential type \cite{GSZ13}.

\begin{Conj}[{\cite[Conjecture 4.7]{GSZ13}}]\label{Ex: list of OPIs of diff  type}
	 The OPIs  appearing  in the list below are all the OPIs of differential type:
	\begin{itemize}
		\item[($\calu$1)] $ \lfloor x y\rfloor-a(x\lfloor y\rfloor+\lfloor x\rfloor y)-b\lfloor x\rfloor\lfloor y\rfloor-cxy$ with $a^{2}=a+bc$,
		\item[($\calu$2)] $\lfloor x y\rfloor-a b^{2} y x-b x y-a\lfloor y\rfloor\lfloor x\rfloor+ab(y\lfloor x\rfloor+\lfloor y\rfloor x)$,
		\item[($\calu$3)] $\lfloor x y\rfloor-x\lfloor y\rfloor-a(x\lfloor 1\rfloor y-\lfloor 1\rfloor x y)$,
		\item[($\calu$4)]$\lfloor x y\rfloor-\lfloor x\rfloor y-a(x\lfloor 1\rfloor y-x y\lfloor 1\rfloor)$,
		\item[($\calu$5)]  $\lfloor x y\rfloor-x\lfloor y\rfloor-\lfloor x\rfloor y-a x\lfloor 1\rfloor y-b x y$,
		\item[($\calu$6)] $\lfloor x y\rfloor-\sum_{i, j \geq 0} \lambda_{i j}\lfloor 1\rfloor^{i} x y\lfloor 1\rfloor^{j}$ with the convention that $\lfloor 1\rfloor^{0}=1$,
	\end{itemize}
where  $a,b,c,\lambda_{ij}\in\bk,~i, j \geq0$.
 %Moreover, any OPI $\phi$ of differential type is necessarily defined as above by a $N(x, y)$ from this list.
\end{Conj}
Guo, Sit and Zheng   proved this conjecture under the  condition $N(x, y)\in \frakM_2(X)$  in \cite{GSZ13}.

In the following result, we will replace each OPI (except possibly Type $\left(\calu6\right)$)  defined in Conjecture~\ref{Ex: list of OPIs of diff  type}   by an equivalent one in the sense of Definition~\ref{Def: equivalent OPIs}.
Notice that we  only consider   $\left(\calu6\right)$  type in the case $\lambda_{ij}=0$ unless $i+j\leq1$.

\begin{thm}\label{Lemma: equivalent OPIs of unital diff type}
	 In  the following statements, $\lambda, \mu, \nu  \in \bk$ are  scalars:
	\begin{itemize}
\item[(a)]  The  OPI ($\calu$1) is equivalent to one of
 $$\begin{array}{ll} (\calu1'_a)  &\lfloor x\rfloor\lfloor y\rfloor-\lambda (x\lfloor y\rfloor+\lfloor x\rfloor y)+ \mu\lfloor xy\rfloor+\nu xy, \ \mathrm{where}\   \lambda^{2}=\lambda \mu+\nu, \\
    	 (\calu1'_b) & x\lfloor y\rfloor-\lfloor xy\rfloor+\lfloor x\rfloor y-\lambda x y,	\\	
		 (\calu1'_c) & \lfloor x\rfloor-\lambda x;
\end{array}$$

\item[(b)] the  OPI ($\calu$2) is equivalent to one of
$$\begin{array}{ll}
		 (\calu2'_a)   &\lfloor x\rfloor\lfloor y\rfloor -\lambda(x\lfloor y\rfloor+\lfloor x\rfloor y)+\mu\lfloor yx\rfloor+\lambda^2 x y-\lambda \mu y x,\\
		(\calu2'_b)   &\lfloor x\rfloor-\lambda x;
\end{array}$$
\item[(c)] the  OPIs ($\calu$3)  and  ($\calu$4) are equivalent   to
$$\begin{array}{ll}
 (\calu 3^\prime) =( \calu 4 ^\prime ) & \lfloor x\rfloor+\lambda x\lfloor 1\rfloor-(\lambda+1)\lfloor 1\rfloor x;
	
\end{array}$$
\item[(d)] the  OPI ($\calu$5) is equivalent to one of
$$\begin{array}{ll}
	 ( \calu5'_a )   &x\lfloor y\rfloor-\lfloor xy\rfloor+\lfloor x\rfloor y-\lambda x y,\\
		 ( \calu5'_b )   &x\lfloor y\rfloor-\lfloor xy\rfloor+\lfloor x\rfloor y-x\lfloor 1\rfloor y,\\
		 ( \calu5'_c )   &x;
\end{array}$$
\item[(e)] the  OPI ($\calu$6) is equivalent to one of
$$\begin{array}{ll}
    	  (\calu6'_a )   &\lfloor x\rfloor+\lambda x\lfloor 1\rfloor-(\lambda+1)\lfloor 1\rfloor x,\\
    	 ( \calu6'_b )  &\lfloor x\rfloor-\lambda x,\\
		 ( \calu6'_c )  &x.
\end{array}$$
\end{itemize}
\end{thm}

\begin{proof}
 %($\calu1\Leftrightarrow\calu1'a,b,c$)
	(a) For OPI ($\calu$1), if $b=0$, we have $a^2=a$. If $a=0$, we get ($\calu1'_c$) with $\lambda=c$; if $a=1$, we get ($\calu1'_b$) still with $\lambda=c$.

Now we assume $b\neq0$.  Then ($\calu1$) is equivalent to
	$$\lfloor x\rfloor\lfloor y\rfloor+\frac{a}{b}(\lfloor x\rfloor y+x\lfloor y\rfloor)-\frac{1}{b}\lfloor xy\rfloor+\frac{c}{b}xy,$$
	with $a^{2}=a+bc$. Replacing $\frac{a}{b},\frac{1}{b}$ and $\frac{c}{b}$ by $-\lambda,-\mu$ and $\nu$ respectively,   we get ($\calu1'_a$).
	\smallskip
	
 %($\calu2\Leftrightarrow\calu2'a,b$)
	(b) For OPI ($\calu2$), if $a=0$, we get ($\calu2'_b$) by replacing $b$ with $\lambda$, which is the same as ($\calu1'_c$). If $a\neq0$, then ($\calu2$) is equivalent to
	$$\lfloor x\rfloor\lfloor y\rfloor -b(x\lfloor y\rfloor+\lfloor x\rfloor y)-\frac{1}{a}\lfloor yx\rfloor+b^2 x y+\frac{b}{a} y x.$$
	Replace $b$ and $-\frac{1}{a}$   by $\lambda$ and $ \mu$   respectively, then we get ($\calu2'_a$).
	\smallskip
	
 %($\calu3\Leftrightarrow\calu3'$ and $\calu4\Leftrightarrow\calu4'$)
	(c) For OPI ($\calu4$), take $x=1$ and replace $y$ by $x$, we get the OPI ($\calu4'$):
		$$
			\lfloor x\rfloor +ax\lfloor 1\rfloor -(a+1)\lfloor 1\rfloor x.	$$
%with $\lambda=a$.
	Conversely,  given ($\calu4'$), by replacing $x$ by $xy$, we get
	$$
		\lfloor xy\rfloor+axy\lfloor 1\rfloor -(a+1)\lfloor 1\rfloor xy,
		$$
	and by multiplying ($\calu4'$)   by $y$ on the right, we get
		$$
	\lfloor x\rfloor y +ax\lfloor 1\rfloor y-(a+1)\lfloor 1\rfloor xy;$$
the difference of  the above two formulae  is just the OPI ($\calu4$):
	$$
	\lfloor xy\rfloor -\lfloor x\rfloor y -a(x\lfloor 1\rfloor y-xy\lfloor 1\rfloor).$$
	So ($\calu$4) and ($\calu$4$^\prime$) are equivalent.

 Similarly, the OPI ($\calu$3) is equivalent to the OPI ($\calu4'$).

 %($\calu5\Leftrightarrow\calu5'a,b,c$)
	(d) 	For the OPI ($\calu$5), take $y=1$, then we get
	\begin{align}
		\label{tilde5}\tag{R1}
		(a+1) x\lfloor 1\rfloor +b x .
	\end{align}

	If $a= -1$ and $b\neq 0$, the OPI  \eqref{tilde5} is reduced  to $x$ which is the OPI ($\calu5'_c$) .  It is obvious that the OPI ($\calu5'_c$)  induces ($\calu5$).

	If $a= -1$ and $b= 0$,  the original OPI becomes
	$$\lfloor xy\rfloor-x\lfloor y\rfloor-\lfloor x\rfloor y+x\lfloor 1\rfloor y,$$
which is exactly ($\calu5'_b$).

	Now we assume $a\neq-1$, then  $$(\calu5)+\frac{a}{a+1}\eqref{tilde5}\cdot y=\lfloor xy\rfloor-x\lfloor y\rfloor-\lfloor x\rfloor y+\lambda x y,$$
	where $\lambda =-\frac{b}{a+1}$, which   gives ($\calu5'_a$). The other direction follows from the fact that
by imposing $y=1$, ($\calu5'_a$) becomes $  x\lfloor 1\rfloor -\lambda x $, which is exactly \eqref{tilde5}.
We have shown that  for $a\neq-1$, ($\calu5$) is equivalent to ($\calu5'_a$).
	\smallskip

 %($\calu6\Leftrightarrow\calu6'$)
	(e) For the OPI $(\calu6)$, we only consider  the case $\lambda_{ij}=0$ unless $i+j\leq1$.

 Now the OPI $(\calu6)$ $$\lfloor x y\rfloor- \lambda_{10}\lfloor 1\rfloor x y- \lambda_{01}x y\lfloor 1\rfloor- \lambda_{00} xy$$ is equivalent to $$\lfloor x\rfloor- \lambda_{10}\lfloor 1\rfloor x- \lambda_{01}x\lfloor 1\rfloor- \lambda_{00} x.$$ Take $x=1$, then we get
	\begin{align}
		\label{tilde6}\tag{R2}
		(1-\lambda_{10}-\lambda_{01})\lfloor 1\rfloor- \lambda_{00}.
	\end{align}
	If $\lambda_{10}+\lambda_{01}\neq 1$, by taking $\lambda=\frac{\lambda_{00}}{1-\lambda_{01}-\lambda_{10}}$, \eqref{tilde6} is equivalent to $\lfloor 1\rfloor- \lambda$. Now it is easy to see $(\calu6)$ is equivalent to $\lfloor x\rfloor- \lambda x$, which is the OPI ($\calu6'_b$).

	If $\lambda_{10}+\lambda_{01}=1$ and $\lambda_{00}=0$, by taking $\lambda_{10}=-\lambda$, we have $\lambda_{01}=\lambda+1$ and $$\lfloor x\rfloor- \lambda_{10}\lfloor 1\rfloor x- \lambda_{01}x\lfloor 1\rfloor- \lambda_{00} x=\lfloor x\rfloor+\lambda\lfloor 1\rfloor x- (\lambda+1)x\lfloor 1\rfloor,$$	 which is the OPI ($\calu6'_a$).

	If $\lambda_{10}+\lambda_{01}=1$ and $\lambda_{00}\neq0$, the OPI $(\calu6)$ is equivalent to  $x$ which is  ($\calu6'_c$) .
\end{proof}

\subsection{Nonunital differential type OPIs}\

We will also consider the nonunital version of differential OPIs and the corresponding  classification problem.

\begin{defn}
	An OPI $\phi\in \bk \frakS(X )$ is said to be of nonunital differential type if $\phi$ is of the form $\lfloor x y\rfloor-N(x, y)$, where $N(x, y)$ satisfies the following conditions:
	\begin{itemize}
		\item[(a)] $N(x, y)$ is linear in $x$ and $y$;
		%in the sense that the total degree of $\lfloor x\rfloor^{n},~n \geq0$ (resp. $\lfloor y\rfloor^{n}$, $n \geq 0$ ) in each monomial of $N(x, y)$ is one
		\item[(b)] no monomial of $N(x, y)$ contains any subword of the form $\lfloor uv\rfloor$ for any $u,v\in\frakS(X) $;%N(x,y) is in $\phi(x,y)$-normal form
		\item[(c)] for any set $Z$ and $u, v, w \in\frakS(Z) $, $$N(u v, w)-N(u, v w)\stackrel{\ast}\rightarrow_{\Pi_{\phi}(Z)}0,$$
		where $\Pi_{\phi}(Z):= \left\{q|_{\lfloor uv\rfloor}\rightarrow q|_{N(u,v)}~|~u,v \in \frakS(Z),~q\in\frakS^\star(Z)\right\}$.
	\end{itemize}
	%If $\phi:=\lfloor x y\rfloor-N(x, y)$ is an OPI of differential type, we also say that the expression $N(x, y)$ and the defining operator $P$ of a $\phi$ -algebra $R$ are of differential type.
\end{defn}

Since any  nonunital differential type OPI  can be seen as a unital differential type OPI, we can also propose the following conjecture:
\begin{Conj} \label{Ex: list of OPIs of nonunital diff  type}
	Each nonunital differential type OPI $\phi$  is one of the OPIs in the following list:
	%Moreover, any OPI $\phi$ of differential type is necessarily defined as above by a $N(x, y)$ from this list.
	
	\begin{itemize}
		\item[($\caln$1)] $\lfloor x y\rfloor-a(x\lfloor y\rfloor+\lfloor x\rfloor y)-b\lfloor x\rfloor\lfloor y\rfloor-cxy$ where $a^{2}=a+bc$,
		\item[($\caln$2)] $\lfloor x y\rfloor-a b^{2} y x-b x y-a\lfloor y\rfloor\lfloor x\rfloor+ab(y\lfloor x\rfloor+\lfloor y\rfloor x)$,
		\item[($\caln$3)]$\lfloor x y\rfloor-x\lfloor y\rfloor $,
		\item[($\caln$4)]  $\lfloor x y\rfloor-\lfloor x\rfloor y $,
	\end{itemize}
where  $a,b,c,\lambda_{ij}\in\bk,~i, j \geq0$.
\end{Conj}

\begin{remark} Notice that the cases ($\caln$1) (resp. ($\caln$2)) are exactly ($\calu$1) (resp. ($\calu$2)).
	The cases ($\caln$3) (resp. ($\caln$4)) are obtained from  ($\calu$3) (resp. ($\calu$4)) by deleting monomials containing  $\lfloor 1\rfloor$.
	However, when all $\lfloor 1\rfloor$ disappear, the cases ($\caln$5) and ($\caln$6)  become    special cases of ($\caln$1).
	
\end{remark}

 We also have a nonunital version of Theorem~\ref{Lemma: equivalent OPIs of unital diff type}, whose proof is similar and omitted. Notice that there is some difference in each case, since we can not take $x$ or $y$ to be $1$.

\begin{thm}\label{Lemma: equivalent OPIs of nonunital diff type}
	We have the following statements:
	\begin{itemize}
		\item[(a)]  The  OPI ($\caln$1) is equivalent to one of
		$$\begin{array}{ll} (\caln1'_a)  &\lfloor x\rfloor\lfloor y\rfloor-\lambda (x\lfloor y\rfloor+\lfloor x\rfloor y)+ \mu\lfloor xy\rfloor+\nu xy, \ \mathrm{where}\   \lambda^{2}=\lambda \mu+\nu, \\
			(\caln1'_b) & x\lfloor y\rfloor-\lfloor xy\rfloor+\lfloor x\rfloor y-\lambda x y,	\\	
			(\caln1'_c) & \lfloor xy\rfloor-\lambda xy;
		\end{array}$$
		
		\item[(b)] the  OPI ($\caln$2) is equivalent to one of
		$$\begin{array}{ll}
			(\caln2'_a)   &\lfloor x\rfloor\lfloor y\rfloor -\lambda(x\lfloor y\rfloor+\lfloor x\rfloor y)+\mu\lfloor yx\rfloor+\lambda^2 x y-\lambda \mu y x,\\
			(\caln2'_b)   &\lfloor xy\rfloor-\lambda xy;
		\end{array}$$
		\item[(c)] the  OPI ($\caln$3)  is equivalent   to
		$$\begin{array}{ll}
			(\caln 3^\prime) & x\lfloor y\rfloor-\lfloor x y\rfloor;
		\end{array}$$
			\item[(d)] the  OPI ($\caln$4) is equivalent   to
	$$\begin{array}{ll}
		( \caln 4 ^\prime ) & \lfloor x\rfloor y-\lfloor x y\rfloor.
	\end{array}$$
	\end{itemize}
\end{thm}

\section{Operated GS bases of free nounital differential type algebras} \label{Section: Free nonunital differential type algebras over algebras}

In this  section, we   apply  Theorem~\ref{Thm: GS basis for free nonunital Phi algebra over nonunital alg}   to OPIs displayed in Theorem~\ref{Lemma: equivalent OPIs of nonunital diff type},  under the monomial order $\leq_{\dl}$, $\leq_{\dl'}$ or  $\leq_{\db}$, we   obtain an operated GS basis as well as a linear basis  for  the free nonunital   differential type algebra over a nonunital   $\bf{k}$-algebra.
However, this method can not be applied to OPIs of type $(\caln1'_c)=(\caln2'_b)$ in Theorem~\ref{Lemma: equivalent OPIs of nonunital diff type}, so we deal with this case directly; see Theorem~\ref{thm:GS basis of nonunital phi4} and Corollary~\ref{Cor: linear basis for type phi4}.

\iffalse
\begin{lem}\label{Lemma: All the intersections of diff OPI are inclusion}
	Let $\phi\in \bk\frakS(\left\lbrace x,y \right\rbrace )$ be a nonunital differential OPI in example \ref{Ex: list of OPIs of nonunital diff  type} with the leading term $x \lfloor y\rfloor$ and respect to the monomial order $\leq_{\mathrm{dl}}$.  If every inclusion composion pair $\left( f,g\right) ^{q}_w$ is trivial modulo $(S_{\phi}(\frakS(Z)) ,w)$, where $f,g\in S_{\phi}(\frakS(Z)) $, then $S_{\phi}(\frakS(Z))$ is a
	GS  in $\bk\frakS(Z)$ with respect to $\leq_{\mathrm{dl}}$.
\end{lem}

\begin{proof}
	Let $\left( \phi(u_1,u_2),\phi(v_1,v_2)\right) ^{a,b}_w$ be a intersection, where $a,b,w\in \frakS(Z) $,     $\phi(u_1,u_2),\phi(v_1,v_2)\in S_{\Phi}(\frakS(Z))$ and  $w=\overline{\phi(u_1,u_2)}a=b\overline{\phi(v_1,v_2)}$. By Lemma
	\ref{Lemma: inequal  relation} and example \ref{Ex: list of OPIs of nonunital diff  type}, we have $\overline{\phi(u_1,u_2)}=u_1 \lfloor u_2\rfloor$ and $\overline{\phi(v_1,v_2)}=v_1 \lfloor v_2\rfloor$. Thus $a$ has a right factor $\lfloor v_2\rfloor$, we can write $w=s\lfloor v_2\rfloor=\overline{\phi(s,u_2)}$, where $s\in \frakS(Z) $. Since the inclusion pair $\left( \phi(s,v_2),\phi(u_1,u_2)\right) ^{\star a}_w$ and $\left( \phi(s,v_2),\phi(v_1,v_2)\right) ^{b\star}_w$   are trivial modulo $(S_{\phi}(\frakS(Z)) ,w)$. Hence,  $\left( \phi(u_1,u_2),\phi(v_1,v_2)\right) ^{a,b}_w$ is trivial modulo $(S_{\phi}(\frakS(Z)) ,w)$.
	
\end{proof}
\fi

\

%The following Lemma and Proposition show that for the nonunital differential OPI $\phi$ given in example \ref{Ex: list of OPIs of nonunital diff  type}, the corresponding set $S_{\phi}(\frakS(Z))$  is an operated GS basis, whose leading term  is $x \lfloor y\rfloor$, with respect to the monomial order $\leq_{\mathrm{dl}}$.

%\begin{lem}\label{Lemma: Nonunital differential OPI is GS}
%	Let $Z$ be a set and $\phi(x,y)=x \lfloor y\rfloor+(-R(x,y))$ be a nonunital differential OPI given in example \ref{Ex: list of OPIs of nonunital diff  type}, whose leading term is $x \lfloor y\rfloor$ with respect to the monomial order $\leq_{\mathrm{dl}}$. Then  $R(uv,w)\downarrow_{\phi}uR(v,w)$, for any $u,v,w\in\frakS(Z)$.
%\end{lem}

\subsection{Nonunital GS OPIs of differential type} \

We first deal with the   types $(\caln1'_b)$  and $(\caln3')$. Our result reads as follows:
\begin{prop}\label{Lemma: Nonunital differential OPI is gronber basis in  free nonunital operated algebra}
	Let $\phi\in \bk\frakS(X)$ be a nonunital differential OPI of type $(\caln1'_b)$ or $(\caln3')$ in Theorem~\ref{Lemma: equivalent OPIs of nonunital diff type}. Let $Z$ be a set.  Then $S_{\phi}(\frakS(Z))$ is an operated
	GS basis in $\bk\frakS(Z)$ with respect to $\leq_{\mathrm{dl}}$.
\end{prop}
 We deduce this result  from Proposition~\ref{Proposition: Nonunital differential OPI xdy is GSB in  free nonunital operated algebra} and Lemma~\ref{Lemma: Nonunital differential OPI is GS}.

For an OPI $\phi$ of type $(\caln1'_b)$ or $(\caln3')$ appearing in  Theorem~\ref{Lemma: equivalent OPIs of nonunital diff type},  we may write $$\phi(x,y)=x \lfloor y\rfloor-R(x,y).$$  By Theorem~\ref{Thm:The well order udl is a monomial order} and Proposition~\ref{Lemma: inequal  relation},
  for each set $Z$ and   for all $u,v\in \frakS (Z)$,  $\phi(u,v) $   vanishes or its  leading monomial is $u\lfloor v\rfloor$ with respect to the monomial order $\leq_{\mathrm{dl}}$ on $\frakS(Z)$.
%The following Proposition and Lemma show that for these two cases, the corresponding set $S_{\phi}(\frakS(Z))$  is an operated GS basis with respect to $\leq_{\mathrm{dl}}$.

\begin{prop}\label{Proposition: Nonunital differential OPI xdy is GSB in  free nonunital operated algebra}
	Let $\phi=x \lfloor y\rfloor -R(x,y) \in\bk \frakS(X )$ be an OPI.
Let $Z$ be a set such that $\frakS (Z)$ is  endowed with a monomial order.  Assume that for all $u,v\in \frakS (Z)$,  $\phi(u,v) $   vanishes or its  leading monomial is $u\lfloor v\rfloor$.  Denote the rewriting system  $$\Pi_{\phi}(Z):= \{q|_{u\lfloor v\rfloor}\rightarrow q|_{R(u,v)}~|~u,v \in \frakS(Z),~ q\in\frakS^\star(Z)\}.$$ If
	$R(uv,w)\downarrow_{\Pi_{\phi}(Z)}uR(v,w)$, for any $u,v,w\in\frakS(Z)$, then $S_{\phi}(\frakS(Z))$ is an operated	GS basis in $\bk\frakS(Z)$ with respect to the monomial order.
\end{prop}

\begin{proof}
We prove the statement by checking all inclusion compositions and intersection compositions are trivial modulo $(S_{\phi}(\frakS(Z)) ,w)$.

We firstly consider inclusion compositions. Let $f=\phi(u,v),g\in  S_{\phi}(\frakS(Z))$ for $u,v\in\frakS(Z)$, such that $w=\bar f=q|_{\bar{g}} $ for some $w\in\frakS(Z)$ and $q\in\frakS^\star(Z)$. We have the following three cases:
\begin{itemize}
	\item [(a)] If $q$ has a right factor $\lfloor v'\rfloor$ with $v'\in\frakS(Z) $. Since $\bar f=u\lfloor v\rfloor$, we have $v=v'$ and thus
	there  exists $q'\in \frakS^\star(Z)$ such that $q=q'\lfloor v\rfloor$ and $ u=q'|_{\bar{g}} $.    So we have
	$$\begin{aligned}
	\left( f,g\right) ^{q}_w&=f-q|_{g}\\
	&=\phi(u,v)-\phi(q'|_{g},v)+\phi(q'|_{g},v)-q|_{g}\\
	&=\phi(q'|_{\bar{g}},v)-\phi(q'|_{g},v)+q'|_{g}\lfloor v\rfloor-R(q'|_{g},v)-q|_{g}\\
	&=\phi(q'|_{\bar{g}-g},v)-R(q'|_{g},v).
	\end{aligned}$$
  By the monomial order, we have $\overline{\phi(q'|_{\bar{g}-g},v)}<\overline{\phi(q'|_{\bar{g}},v)}=w$ and  $\overline{R(q'|_{g},v)}<\overline{\phi(q'|_{g},v)}=w$. Then since $\phi(q'|_{\bar{g}-g},v)$ is trivial modulo $(S_{\phi}(\frakS(Z)) ,w)$ and  $g\in S_{\phi}(\frakS(Z))$,  the inclusion composition $\left( f,g\right) ^{q}_w$
   is trivial modulo $(S_{\phi}(\frakS(Z)) ,w)$.
  \item [(b)] If $q$ has a right factor $\lfloor v'\rfloor$ with $v'\in\frakS^\star(Z)$. Since $\bar f=u\lfloor v\rfloor$, we have $v=v'|_{\bar{g}}$ and $q=u\lfloor v'\rfloor$, thus
   	$$\begin{aligned}
   	\left( f,g\right) ^{q}_w&=f-q|_{g}\\
   &=\phi(u,v'|_{\bar{g}})-u\lfloor v'|_{g}\rfloor\\
   	&=\phi(u,v'|_{\bar{g}})-\left( u\lfloor v'|_{g}\rfloor-R(u,v'|_{g})\right)-R(u,v'|_{g}) \\
   	&= \phi(u,v'|_{\bar{g}})- \phi(u,v'|_{g})-R(u,v'|_{g})\\
   	&=\phi(u,v'|_{\bar{g}-g})-R(u,v'|_{g}),
   \end{aligned}$$
   which is trivial modulo $(S_{\phi}(\frakS(Z)) ,w)$.

	\item [(c)] If $q$ has no right factor of the form $\lfloor v'\rfloor$, then $\star$ is a right factor of $q$ and $\overline{g}$ has the right factor $\lfloor v\rfloor$, i.e., $q=t\star$ and $\overline{g}=s\lfloor v\rfloor$ for some $t,s\in\frakS(Z)$. So we have $t\overline{g}=\overline{f}$ and $g=\phi(s,v)$.  By assumption, we have $R(ts,v)\downarrow_{\Pi_{\phi}(Z)}tR(s,v)$ and by Lemma~\ref{Lemma:joinable gives trivial modulo} the inclusion composition  $\left( f,g\right) ^{q}_w$ is trivial modulo $(S_{\phi}(\frakS(Z)) ,w)$.
\end{itemize}

%If every inclusion composition pair $\left( f,g\right) ^{q}_w$ is trivial modulo $(S_{\phi}(\frakS(Z)) ,w)$, where $f,g\in S_{\phi}(\frakS(Z)) $,  $w=\bar{f}=q|_{\bar{g}}\in \frakS(Z)$ , $q\in \frakS^\star(Z)$. Then $S_{\phi}(\frakS(Z))$ is a
%GS  in $\bk\frakS(Z)$ with respect to $\leq_{\mathrm{dl}}$.	
Next we consider the intersection compositions. Given $f=\phi(u_1,u_2)$ and $g=\phi(v_1,v_2)$ with $u_1,u_2,v_1,v_2\in \frakS(Z)$, assume there exist $a,b,w\in \frakS(Z) $ such that $w=\overline{f}a=b\overline{g}$, i.e., $w=u_1 \lfloor u_2\rfloor a=bv_1 \lfloor v_2\rfloor$ and max$\left\lbrace |\bar{f}| ,|\bar{g}|   \right\rbrace < \left| w\right| < |\bar{f}| +|\bar{g}|$.
We will write the intersection composition $\left(f,g\right) ^{a,b}_w$ into two inclusion compositions.
Consider the polynomial $h=\phi(bv_1,v_2)$. We have $w=\bar h=q_1|_{\bar f}=q_2|_{\bar g}$, where $q_1=\star a$ and $q_2=b\star$. As we have shown above, the inclusion compositions $\left( h,f\right) ^{\star a}_w$ and $\left( h,g\right) ^{b\star}_w$   are trivial modulo $(S_{\phi}(\frakS(Z)) ,w)$.
And since $$\left(f,g\right) ^{a,b}_w=fa-bg=-(h-fa)+(h-bg)=-\left( h,f\right) ^{\star a}_w+\left( h,g\right) ^{b\star}_w,$$
the intersection composition $\left(f,g\right) ^{a,b}_w$ is trivial modulo $(S_{\phi}(\frakS(Z)) ,w)$. 	
\end{proof}

We need a technical lemma which enables us to apply    Proposition~\ref{Proposition: Nonunital differential OPI xdy is GSB in  free nonunital operated algebra}.
\begin{lem}\label{Lemma: Nonunital differential OPI is GS}
	Let  $\phi(x,y)=x \lfloor y\rfloor-R(x,y)$ be a nonunital differential OPI of type $(\caln1'_b)$ or $(\caln3')$ given in Theorem~\ref{Lemma: equivalent OPIs of nonunital diff type}.  Let $Z$ be a set. Consider  the rewriting system $$\Pi_{\phi}(Z):= \{q|_{u\lfloor v\rfloor}\rightarrow q|_{R(u,v)}~|~u,v \in \frakS(Z),~ q\in\frakS^\star(Z)\}.$$ Then  $R(uv,w)\downarrow_{\Pi_{\phi}(Z)}uR(v,w)$, for any $u,v,w\in\frakS(Z)$.
\end{lem}

\begin{proof}
	  For $\phi$ of type $(\caln1'_b)$, the OPI $\phi$ can be written as $x\lfloor y\rfloor-R(x,y)$, where $R(x,y)=\lfloor xy\rfloor-\lfloor x\rfloor y+cx y$.
	We have
   	$$\begin{aligned}
		uR(v,w)
		&=~u\lfloor vw\rfloor-u\lfloor v\rfloor w+cuvw\\
		&\rightarrow\lfloor uvw\rfloor-\lfloor u\rfloor vw+cuvw-(\lfloor uv\rfloor  -\lfloor u\rfloor v +cuv )w+cuvw\\
		&=~\lfloor uvw\rfloor-\lfloor uv\rfloor w+cuvw\\
		&=~R(uv,w).
   \end{aligned}$$
	
	  For the type $(\caln3')$ in Theorem~\ref{Lemma: equivalent OPIs of nonunital diff type}, $\phi(x,y)=x \lfloor y\rfloor-\lfloor xy\rfloor$ with $R(x,y)=\lfloor xy\rfloor$. We have
	$$	
	uR(v,w)=u\lfloor vw\rfloor
	\rightarrow\lfloor uvw\rfloor=R(uv,w).$$

	So in either case, $R(uv,w)\downarrow_{\Pi_{\phi}(Z)}uR(v,w)$, for any $u,v,w\in\frakS(Z)$.
\end{proof}

Proposition~\ref{Proposition: Nonunital differential OPI xdy is GSB in  free nonunital operated algebra} and Lemma~\ref{Lemma: Nonunital differential OPI is GS}  implies immediately Proposition~\ref{Lemma: Nonunital differential OPI is gronber basis in  free nonunital operated algebra}.
This finishes our discussion for the   types $(\caln1'_b)$  and $(\caln3')$.

\medskip

With a   proof similar to that of  Type $(\caln3')$ under the order $\leq_{\mathrm{dl'}}$ given in Remark~\ref{remark: dl'},  one can treat the case $(\caln4')$.
\begin{lem}
	For the OPI $\phi(x,y)=\lfloor x\rfloor y-\lfloor xy\rfloor $ of
	type $(\caln4')$  in Theorem~\ref{Lemma: equivalent OPIs of nonunital diff type},
	   $S_{\phi}(\frakS(Z))$ is an operated	GS basis in $\bk\frakS(Z)$ with respect to the monomial order $\leq_{\mathrm{dl'}}$ for any set $Z$.
\end{lem}

\begin{prop}
	Consider an OPI  $\phi$ of type $(\caln1'_a)$ (resp. $(\caln2'_a)$, $(\caln1'_c)=(\caln2'_b)$). For each set $Z$,   $S_{\phi}(\frakS(Z))$ is an operated	GS basis in $\bk\frakS(Z)$ with respect to $\leq_{\mathrm{db}}$
\end{prop}	
\begin{proof}
 The case  $(\caln1'_a)$ (resp. $(\caln2'_a)$) follows from  Lemma~\ref{Lem: differential type 1a is GS} (resp. Lemma~\ref{Lem: differential type 2a is GS}) and the case $(\caln1'_c)=(\caln2'_b)$ can be obtained  by direct inspection.
\end{proof}
	
We have proved the following result:
\begin{thm}\label{Thm: nonunital differential type are GS}
	Each OPI in Theorem~\ref{Lemma: equivalent OPIs of nonunital diff type} is  nonunital GS with respect to either  $\leq_{\rm{dl}}$, $\leq_{\rm{dl'}}$ or $\leq_{\rm{db}}$.
\end{thm}

\subsection{Operated GS bases of free nonunital differential type algebras over algebras} \

 By Theorem~\ref{Thm: nonunital differential type are GS},  Theorem~\ref{Thm: GS basis for free nonunital Phi algebra over nonunital alg} can be applied directly  to all differential type OPIs  in Theorem~\ref{Lemma: equivalent OPIs of nonunital diff type} (except Type $(\caln1'_c)=(\caln2'_b)$) to obtain operated GS bases for free nonunital  differential algebras over algebras.
\begin{thm}\label{thm: GS basis for nonunital differential type}
	Let $Z$ be a set, $A=\bk \cals(Z)\slash I_A$ an algebra with a GS  basis $G$ with respect to $\leq_{\rm{dlex}}$.
	\begin{itemize}
	 \item[(a)] Let $\phi_1$ be an OPI of type $(\caln1'_a)$ or $(\caln2'_a)$.
	Then $S_{\phi_1}(\frakS(Z))\cup G$ is an operated GS  basis of $\langle S_{\phi_1}(\frakS(Z))\cup I_A\rangle_\OpAlg$ in $\bk\frakS(Z) $ with respect to $\leq_{\rm{db}}$. %And the set $$\Irr(S_{\phi_1}(\frakS(Z))\cup G)=\frakS(Z)\backslash \left\lbrace q|_{\bar{s}},q|_{\lfloor u\rfloor\lfloor v\rfloor}~|~s\in G,q\in\frakS^\star(Z),u,v\in\frakS(Z)\right\rbrace$$ is a basis of the free $\phi_1$-algebra $\calf^{\phi_1\zhx\Alg}_{\Alg}(A)$ over $A$.
	
	\item[(b)] Let $\phi_2$ be an OPI of type $(\caln1'_b)$ or $(\caln3')$.  	Then $S_{\phi_2}(\frakS(Z))\cup G$ is an operated GS  basis of $\langle S_{\phi_2}(\frakS(Z))\cup I_A\rangle_\OpAlg$ in $\bk\frakS(Z) $ with respect to $\leq_{\rm{dl}}$. %And the set $$\Irr(S_{\phi_2}(\frakS(Z))\cup G)=\frakS(Z)\backslash \left\lbrace q|_{\bar{s}},q|_{ u\lfloor v\rfloor}~|~s\in G,q\in\frakS^\star(Z),u,v\in\frakS(Z)\right\rbrace$$ is a basis of the free $\phi_2$-algebra $\calf^{\phi_2\zhx\Alg}_{\Alg}(A)$ over $A$.
		
	\item[(c)] Let $\phi_3$ be an OPI of type $(\caln4')$. 	Then $S_{\phi_3}(\frakS(Z))\cup G$ is an operated GS  basis of $\langle S_{\phi_3}(\frakS(Z))\cup I_A\rangle_\OpAlg$ in $\bk\frakS(Z) $ with respect to $\leq_{\rm{dl'}}$. %And the set	$$\Irr(S_{\phi_3}(\frakS(Z))\cup G)=\frakS(Z)\backslash \left\lbrace q|_{\bar{s}},q|_{\lfloor u\rfloor v}~|~s\in G,q\in\frakS^\star(Z),u,v\in\frakS(Z)\right\rbrace$$   is a basis of the free $\phi_3$-algebra $\calf^{\phi_3\zhx\Alg}_{\Alg}(A)$ over $A$.	
		\end{itemize}
\end{thm}

 %\begin{thm}\label{thm: GS basis for nonunital differential type}
% 	Let $Z$ be a set, $A=\bk \cals(Z)\slash I_A$ an algebra with a GS  basis $G$ with respect to $\leq_{\rm{dlex}}$.
% Let $\phi$ be an OPI of type $(\caln1'_a)$ or $(\caln2'_a)$ (resp. $(\caln1'_b)$ or $(\caln3')$, $(\caln4')$).
 %		Then $S_{\phi}(\frakS(Z))\cup G$ is an operated GS  basis of $\langle S_{\phi}(\frakS(Z))\cup I_A\rangle_\OpAlg$ in $\bk\frakS(Z) $ with respect to $\leq_{\rm{db}}$ (resp. $\leq_{\rm{dl}}$ ,$\leq_{\rm{dl'}}$).
% \end{thm}

Then by Theorem~\ref{Thm: nonunital CD}, we obtain linear bases for these algebras.
In the statement of the following result, write $\lfloor u\rfloor^{(1)}=\lfloor u\rfloor$ and $\lfloor u\rfloor^{(k+1)}=\lfloor\lfloor u\rfloor^{(k)}\rfloor$ for $k\geq1$.  For a subset $G\subset \cals(Z)$, we also introduce two notations
$$\Irr_{\cals}(G):=\cals(Z)\backslash\lbrace u\bar{g}v~|~g\in G, u,v\in\calm(Z)\rbrace$$
and
$$\Irr_{\calm}(G):=\calm(Z)\backslash\lbrace u\bar{g}v~|~g\in G, u,v\in\calm(Z)\rbrace.$$
Note that $\Irr_{\calm}(G)=\Irr_{\cals}(G)\cup \{1\}$.
In the case that $G$ is a GS basis of $\cals(Z)$ (resp.$\calm(Z) $),  $\Irr_{\cals}(G)$ (resp. $\Irr_{\calm}(G)$) is a linear basis of
$A=\bk \cals(Z)/(G)$  (resp. $A=\bk \calm(Z)/(G)$).

\begin{thm}\label{thm: linear basis of nonunital differential type algebras over algebras}
	With the same setup as in Theorem~\ref{thm: GS basis for nonunital differential type},  we have the following statements:
\begin{itemize}
	\item[(a)] 	Let $B_0=\Irr(G)$ and for   $n\geq 1$, define the set $B_{n}$  inductively as
	$$
		B_0\bigcup  \bigcup_{r\geq1}\Big\lbrace  u_{0}\lfloor v_{1}\rfloor u_{1}    \cdots\lfloor v_{r} \rfloor u_{r}~\big|~u_0,u_r\in\Irr_{\calm}(G),
		 	  u_i\in\Irr_{\cals}(G),1\leq i\leq r-1,  v_j\in B_{n-1}, 1\leq j\leq r  \Big\rbrace .  	$$
%		$$B_{n}=B_0\bigcup\left\lbrace {\Irr^\calm(G)\lfloor B_{n-1}\rfloor \Irr^\cals(G)    \cdots\Irr^\cals(G)\lfloor B_{n-1} \rfloor \Irr^\calm(G)}\right\rbrace,
%	$$
	Then the set
	$B= \bigcup_{n\geq0}B_n$ is a  linear basis of the free $\phi_1$-algebra $\calf^{\phi_1\zhx\Alg}_{\Alg}(A)$ over $A$.

	\item[(b)] The union of $\Irr(G)$ with
$$  \bigcup_{r\geq1}\left\lbrace {\lfloor\lfloor\cdots\lfloor\lfloor u_{r}\rfloor^{(k_r)} u_{r-1}\rfloor^{(k_{r-1})} \cdots u_{1}\rfloor^{(k_1)}  u_{0}~\big|~u_{0}\in\Irr_{\calm}(G), u_i\in\Irr_{\cals}(G), k_i\geq1, 1\leq i\leq r}\right\rbrace
	 $$
	is a linear basis of the free $\phi_2$-algebra $\calf^{\phi_2\zhx\Alg}_{\Alg}(A)$ over $A$.
	
	\item[(c)] The union of $\Irr(G)$ with
	$$\bigcup_{r\geq1}\left\lbrace { u_{0} \lfloor u_{1}  \cdots \lfloor u_{r-1}\lfloor u_{r}\rfloor^{k_r}\rfloor^{k_{r-1}}\cdots\rfloor^{k_1}~\big|~u_{0}\in\Irr_{\calm}(G), u_i\in\Irr_{\cals}(G), k_i\geq1, 1\leq i\leq r}\right\rbrace$$
	is a linear  basis of the free $\phi_3$-algebra $\calf^{\phi_3\zhx\Alg}_{\Alg}(A)$ over $A$.
\end{itemize}		
\end{thm}
The proof is straightforward and is left to the reader.

\begin{remark}
	Notice that the linear basis of a free nonunital differential type algebra  over a nonunital algebra  given in  Theorem~\ref{thm: linear basis of nonunital differential type algebras over algebras} (a) has the same form as that of  free nonunital Rota-Baxter type algebras over algebras, because  the corresponding GS bases share the same set of leading monomials,  consisting of all $\lfloor u\rfloor\lfloor v\rfloor$ and $\bar g$ with $u,v\in \frakS(Z), g\in G$. The phenomenon also occurs for the unital case; see Theorem~\ref{thm: linear basis of unital differential type algebra over algebra} (a) and  \cite[Theorem~6.7]{QQWZ}.
	
	As a special case of Rota-Baxter type, the same  construction of linear bases for the   free  nonunital Nijenhuis algebras over algebras appeared in \cite{LeiGuo}.
\end{remark}

Notice that one cannot  apply Theorem~\ref{Thm: GS basis for free nonunital Phi algebra over nonunital alg} to  $\lfloor xy\rfloor-\lambda xy$ which is of type $(\caln1'_c)=(\caln2'_b)$ in Theorem~\ref{Lemma: equivalent OPIs of nonunital diff type}, as each term   has a subword $xy$. Nevertheless,
we could deal  with  this case  directly.

We need to introduce the  notions of linearly self-reduced sets and bases.

\begin{defn}[{\cite[Definition~1.2.1.2]{BremnerDotsenko}}]
	Let $V$ be a vector space with a chosen   well ordered basis. Each nonzero element of $V$, expressed as  a linear combination of basis elements,  has a leading basis element   which  is maximal among all basis elements having nonzero coefficients in this linear combination.

Let $S$ be a subset of   $V$.
 A basis element  is said to be linearly reduced with respect to $S$ if  it is not a leading basis element of any element of $S$. More generally, an element $f \in V$ is said to be linearly reduced with respect to $S$, if  all  basis monomials  having nonzero coefficients in $f$   are linearly reduced with respect to $S$.
	
	A subset $S \subset V$ is said to be linearly self-reduced if  each element $s \in S$ has  its   leading basis element with  coefficient $1$ and  is linearly reduced with respect to $S \backslash\{s\}$.
\end{defn}

\begin{prop}[{\cite[Proposition~1.2.1.6]{BremnerDotsenko}}]\label{prop: each subspace has a linearly self-reduced basis}
	Let $V$ be a vector space  with a chosen well ordered basis. Every  finite dimensional subspace $S\subset V$ has a unique linearly self-reduced basis.
\end{prop}

Now we can consider OPIs of the type $(\caln1'_c)= (\caln2'_b)$.
\begin{thm}\label{thm:GS basis of nonunital phi4}
	Let $Z$ be a set and  $A=\bk \cals(Z)\slash I_A$ an algebra with a finite  GS  basis $G$ with respect to $\leq_{\rm{dlex}}$.
 Let $\phi_4=\lfloor xy\rfloor-\lambda xy$ be an OPI of type $(\caln1'_c)= (\caln2'_b)$ in Theorem~\ref{Lemma: equivalent OPIs of nonunital diff type}. Denote by $\pi:\bk\cals(Z)\twoheadrightarrow \bk Z \hookrightarrow\bk\cals(Z)$ the canonical projection. Then $\pi(G)$ is just the set of linear parts of elements of $G$.  Let $\bk\pi(G)$ be the subspace of $\bk Z$ spanned by  $\pi(G)$.

 Let $G^{(1)}$ be the unique  linearly self-reduced basis of $\bk\pi(G)$ and define $$\calg'=\lbrace \lfloor g\rfloor-\lambda g ~\big|~ g\in  G^{(1)}  \rbrace.$$
		Then $S_{\phi_4}(\frakS(Z))\cup G\cup \calg'$ is an operated GS  basis of $\langle S_{\phi_4}(\frakS(Z))\cup I_A\rangle_\OpAlg$ in $\bk\frakS(Z) $ with respect to $\leq_{\rm{db}}$.
\end{thm}

\begin{proof}
	We first prove that $S_{\phi_4}(\frakS(Z))\cup G\cup \calg' $ is a subset of $\langle S_{\phi_4}(\frakS(Z))\cup I_A\rangle_\OpAlg$.
	Since $G^{(1)}\subset \bk\pi(G)$, any element in $\calg' =\lbrace \lfloor g\rfloor-\lambda g ~\big|~ g\in  G^{(1)}  \rbrace$ can be written as a linear combination of the elements in the set $\left\lbrace \lfloor\pi(g)\rfloor-\lambda \pi(g) ~|~ g\in G \right\rbrace$.
	And for any $g\in G$,
	$$\lfloor\pi(g)\rfloor-\lambda \pi(g)=\left( \lfloor\pi(g)-g\rfloor-\lambda (\pi(g)-g)\right) +\lfloor g\rfloor-\lambda g\in\langle S_{\phi_4}(\frakS(Z))\cup I_A\rangle_\OpAlg, $$
	since $\lfloor\pi(g)-g\rfloor-\lambda (\pi(g)-g)\in S_{\phi_4}(\frakS(Z)) $ and $\lfloor g\rfloor,g\in \langle I_A\rangle_\OpAlg$. It follows that  $\calg'  \subset\langle S_{\phi_4}(\frakS(Z))\cup I_A\rangle_\OpAlg.$
	
	Now we show $S_{\phi_4}(\frakS(Z))\cup G\cup \calg' $ is an operated GS basis. 	Notice that for any $u,v\in\frakS(Z)$, the leading monomial  of  $\phi_4(u,v)=\lfloor uv\rfloor-\lambda uv$   is $\lfloor uv\rfloor$ with respect to $\leq_{\rm{db}}$.
	For any elements $a_1\in S_{\phi_4}(\frakS(Z))$, $a_2\in G$ and $ a_3\in \calg'$, the leading monomials are of the forms $\overline a_1=\lfloor uv\rfloor$, $\overline a_2=w$ and $\overline a_3=\lfloor z\rfloor$, where $u,v\in\frakS(Z),~w\in G\subset\cals(Z)$ and $z\in\calg'\subset Z$. Hence, these three leading monomials have no overlap with each other.
	And since each of the sets $S_{\phi_4}(\frakS(Z))$, $G$ and $ \calg' $ is an operated GS basis, we only need to check the triviality of any inclusion composition $\left( f,g\right)^q_{w}$    modulo $(S_{\phi_4}(\frakS(Z))\cup G\cup \calg', w)$, for the cases $(f\in S_{\phi_4}(\frakS(Z)),~ g\in G )$,  $(f\in S_{\phi_4}(\frakS(Z)),~ g\in \calg')$ and $(f\in \calg',~ g\in G)$.
	
	(a) For the case $w=\overline{f}=q|_{\overline{g}}$ with $f\in S_{\phi_4}(\frakS(Z)),~ g\in G $. There exists $q'\in\frakS^\star(Z)$ such that $q=\lfloor q'\rfloor$ and $f=\lfloor q'|_{\overline{g}}\rfloor-\lambda q'|_{\overline{g}}$.
	
	If  $g\in  \bk Z$, there exists a factor of $q'$ in  $\cals(Z)$. Then we have
	$$\left( f,g\right)^q_{w}=f-q|_g=\lfloor q'|_{\overline{g}}\rfloor-\lambda q'|_{\overline{g}}-\lfloor  q'|_{ g}\rfloor=(\lfloor q'|_{\overline{g}-g}\rfloor-\lambda q'|_{\overline{g}-g})-\lambda q'|_g.$$
	It is trivial modulo  $(S_{\phi_4}(\frakS(Z))\cup G\cup \calg',w)$ since $\lfloor q'|_{\overline{g}-g}\rfloor-\lambda q'|_{\overline{g}-g}$ is trivial modulo $(S_{\phi_4}(\frakS(Z)) ,w)$ and $g\in G$.
	
	If $g\in G\backslash \bk Z$, it can be  written as $g=g_1+\pi(g)\in \bk(\cals(Z)\backslash Z) \oplus \bk Z$ and we have $\overline g=\overline g_1$. The inclusion composition $\left( f,g\right)^q_{w}$ can be written as
	$$\begin{aligned}
		\left( f,g\right)^q_{w}&=f-q|_g\\
		&=\lfloor q'|_{\overline{g}}\rfloor-\lambda q'|_{\overline{g}}-\lfloor  q'|_{g_1+\pi(g)}\rfloor\\
		&=\lfloor q'|_{\overline{g}}\rfloor-(\lambda q'|_{\overline{g}}+\lambda q'|_{g-\overline g})+\lambda q'|_{g-\overline g}-(\lfloor  q'|_{g_1}\rfloor+\lfloor  q'|_{\pi(g)}\rfloor)\\
		&=(\lfloor q'|_{\overline{g}}\rfloor-\lfloor  q'|_{g_1}\rfloor)-(\lambda q'|_{\overline{g}}+\lambda q'|_{g-\overline g})+\lambda q'|_{g-g_1}+\lambda q'|_{g_1-\overline g}-\lfloor  q'|_{\pi(g)}\rfloor\\
		&=-(\lfloor q'|_{g_1-\overline g}\rfloor-\lambda q'|_{g_1-\overline g})-\lambda q'|_g-(\lfloor q'|_{\pi(g)}\rfloor-\lambda q'|_{\pi(g)}).
	\end{aligned}$$
	Since $\lfloor q'|_{g_1-\bar g}\rfloor -\lambda q'|_{g_1-\bar g}$ is trivial modulo $(S_{\phi_4}(\frakS(Z)) ,w)$ and $g\in G$, the first two terms are trivial modulo $(S_{\phi_4}(\frakS(Z))\cup G\cup \calg',w)$. Now we only need to consider the third term.
	Since $G^{(1)}$ is the linearly self-reduced basis of $\bk\pi(G)$, we have $ \pi(g)=\sum_{i=1}^n k_ig_i$ where $g_i\in G^{(1)}$. Notice that we have $\bar g_i\leq\overline{\pi(g)}$ for $1\leq i\leq n$. Otherwise, there exist $g_i,g_j\in G^{(1)}$ such that $\overline g_i=\overline g_j$, which contradicts to that $G^{(1)}$ is linearly self-reduced. If $q'$ has a factor  in $\cals(Z)$, we have $\lfloor q'|_{\pi(g)}\rfloor-\lambda q'|_{\pi(g)}$ is trivial modulo $(S_{\phi_4}(\frakS(Z)) ,w)$; otherwise, $q'=\star$ and thus $\lfloor q'|_{\pi(g)}\rfloor-\lambda q'|_{\pi(g)}=\sum_{i=1}^n k_i(\lfloor g_i\rfloor -\lambda g_i)$ is trivial modulo $( \calg',w)$. So we have shown that the inclusion composition $\left( f,g\right)^q_{w}$ is trivial modulo  $(S_{\phi_4}(\frakS(Z))\cup G\cup \calg',w)$.

	(b) For the case $w=\overline{f}=q|_{\overline{g}}$ with $f\in S_{\phi_4}(\frakS(Z)),~ g\in \calg'$. There exists $q'\in\frakS^\star(Z)$ such that $q=\lfloor q'\rfloor$ and $f=\lfloor q'|_{\overline{g}}\rfloor-\lambda q'|_{\overline{g}}$. Since $\overline{g}\in \lfloor Z\rfloor$, there exists a factor of $q'$ in  $\cals(Z)$. Thus we have
	$$\left( f,g\right)^q_{w}=f-q|_g=\lfloor q'|_{\overline{g}}\rfloor-\lambda q'|_{\overline{g}}-\lfloor  q'|_{ g}\rfloor=(\lfloor q'|_{\overline{g}-g}\rfloor-\lambda q'|_{\overline{g}-g})-\lambda q'|_g,$$
	which is trivial modulo  $(S_{\phi_4}(\frakS(Z))\cup G\cup \calg',w)$ since $\lfloor q'|_{\overline{g}-g}\rfloor-\lambda q'|_{\overline{g}-g} $ is trivial modulo $(S_{\phi_4}(\frakS(Z)) ,w)$ and $g\in \calg'$.
	
	(c) For the case $w=\overline{f}=q|_{\overline{g}}$ with $f\in \calg',~ g\in G$. Write $f=\lfloor f_1\rfloor-\lambda f_1$, where $f_1\in G^{(1)}$. Since $\overline{f}=q|_{\overline{g}}\in \lfloor Z\rfloor$, we have $q=\lfloor \star\rfloor$, $\bar f_1=\bar g$ and $g=\pi(g)\in \pi(G)$.  Since $G^{(1)}$ is the linearly self-reduced basis of $\bk\pi(G)$, we have $f_1-g=\sum_{i=1}^n k_ig_i$, where $g_i\in G^{(1)}$ and $\bar g_i<\bar f_1$ for $1\leq i\leq n$. So the inclusion composition
	$$\left( f,g\right)^q_{w}=\lfloor f_1\rfloor-\lambda f_1-\lfloor g\rfloor=\left(\lfloor f_1-g\rfloor-\lambda (f_1-g)\right)-\lambda g$$
	is trivial modulo $(S_{\phi_4}(\frakS(Z))\cup G\cup \calg',w)$, since $\lfloor f_1-g\rfloor-\lambda (f_1-g)=\sum_{i=1}^n k_i(\lfloor g_i\rfloor-\lambda g_i)$ is trivial modulo $(\calg',w)$ and $g\in G$.
	
	This  finishes the proof.
	%We have proved that $S_{\phi_4}(\frakS(Z))\cup G\cup \calg'$ is a is a GS basis of $\langle S_{\phi_4}(\frakS(Z))\cup I_A\rangle_\OpAlg$. By Theorem~\ref{Prop: from algebra to Phi algebra with relations} and Theorem~\ref{Thm: nonunital CD}, we obtain a basis of $\calf^{\phi_4\zhx\Alg}_{\Alg}(A)$.
\end{proof}

By Theorem~\ref{Thm: nonunital CD}, we obtain a linear basis.
\begin{thm}\label{Cor: linear basis for type phi4}
	Within the same setup as in Theorem~\ref{thm:GS basis of nonunital phi4},   the union of $ \Irr(G)$ with
		$$\bigcup_{r\geq1}\left\lbrace {u_{0}\lfloor v_{1}\rfloor^{(k_1)} u_{1}    \cdots\lfloor v_{r} \rfloor^{(k_r)} u_{r}~|~u_i\in\Irr(G)\cup\{1\},0\leq i\leq r,  v_j\in Z\backslash \lbrace \bar{g}~|~g\in G^{(1)} \rbrace, k_j\geq1,  1\leq j\leq r }\right\rbrace
	 $$
	  is a linear basis of the free $\phi_4$-algebra $\calf^{\phi_4\zhx\Alg}_{\Alg}(A)$ over $A$.
\end{thm}

\begin{proof}By Theorem~\ref{thm:GS basis of nonunital phi4}, the set $$
		 \Irr(S_{\phi_4}(\frakS(Z))\cup G\cup \calg')=\frakS(Z)\backslash \left\lbrace q|_{\bar{s}}, q|_{\lfloor uv\rfloor}~|~s\in G\cup\calg',q\in\frakS^\star(Z),u,v\in\frakS(Z) \right\rbrace
 $$
 is a linear basis of $\calf^{\phi_4\zhx\Alg}_{\Alg}(A)$.
 Besides  elements of $ \Irr(G)$,  for a monomial $u_{0}\lfloor v_{1}'\rfloor u_{1} \lfloor v_{2}'\rfloor    \cdots\lfloor v_{r}' \rfloor   u_{r}$  with $r\geq 1, u_i\in\Irr(G)\cup\{1\},0\leq i\leq r,  v_j'\in  \cals(Z),   1\leq j\leq r$,
 we should avoid all subwords of the form $\lfloor xy \rfloor, x, y\in \cals(Z)$ and $  \lfloor z \rfloor$ for $z\in Z$ being the leading basis element of an element of $G^{(1)}\subset \bk Z$. The result can be  deduced from this observation.
\end{proof}

\section{Operated GS bases of free unital differential type algebras}

In this section, we will show that under some monomial  order, all unital differential type OPIs in Theorem~\ref{Lemma: equivalent OPIs of unital diff type} are GS  and satisfy the conditions in Theorem~\ref{Thm: GS basis for free unital Phi algebra over unital alg}. As a consequence, we can obtain an operated GS basis and a  linear basis for  each  free unital differential type algebra over a unital algebra with a given GS basis.

\subsection{Unital GS OPIs of differential type} \

\begin{lem}\label{Lem: differential type 1a is GS}
	The OPIs of type $(\calu1'_a)$ in Theorem~\ref{Lemma: equivalent OPIs of unital diff type} are unital   GS with respect to the monomial order $\leq_{\rm{db}}$.
\end{lem}

\begin{proof}
	We will prove the lemma by checking such OPIs satisfy the conditions in  Theorem~\ref{Thm: verify a GS OPI}.
	According to the monomial order $\leq_{\rm{db}}$, the leading monomial of such an OPI is $\lfloor x\rfloor\lfloor y\rfloor$. Consider the rewriting system defined by
	$$\Pi_\phi(Z):=\left\{~\lfloor u\rfloor\lfloor v\rfloor\rightarrow \lambda \lfloor u\rfloor v+\lambda u\lfloor v\rfloor-\mu \lfloor uv\rfloor-\nu  uv~|~u,v\in\frakM(Z)~\right\}.$$
	For the first condition in  Theorem~\ref{Thm: verify a GS OPI}, assume there are $u, v, v^\prime, w, \alpha, \beta, \gamma\in\mathfrak{M}(Z)$, such that
	$\lfloor u\rfloor\lfloor v\rfloor=\alpha\beta$ and $\lfloor v^\prime\rfloor\lfloor w\rfloor=\beta\gamma$. Then we only need to consider the case $\alpha=\lfloor u\rfloor$, $\beta=\lfloor v\rfloor=\lfloor v^\prime\rfloor$ and $\gamma=\lfloor w\rfloor$.
	
	On  one hand, we have
	$$	\begin{array}{lll}\vspace{1ex}
		&R(\phi(u,v)) \gamma \\\vspace{1ex}
		=&\lambda \left\lfloor u\right\rfloor v\left\lfloor w\right\rfloor+\lambda  u\left\lfloor v\right\rfloor\left\lfloor w\right\rfloor-\mu \left\lfloor uv\right\rfloor\left\lfloor w\right\rfloor-\nu uv\left\lfloor w\right\rfloor\\\vspace{1ex}
		\rightarrow&\lambda \left\lfloor u\right\rfloor v\left\lfloor w\right\rfloor+\lambda  u(\lambda \left\lfloor v\right\rfloor w+\lambda  v\left\lfloor w\right\rfloor-\mu \left\lfloor vw\right\rfloor-\nu  vw)\\\vspace{1ex}
		&-b (\lambda \left\lfloor uv\right\rfloor w+\lambda  uv\left\lfloor w\right\rfloor-\mu \left\lfloor uvw\right\rfloor-\nu  uvw)-\nu  uv\left\lfloor w\right\rfloor\\\vspace{1ex}
		=&\lambda \left\lfloor u\right\rfloor v\left\lfloor w\right\rfloor+\lambda ^2 u\left\lfloor v\right\rfloor w+\lambda ^2 uv\left\lfloor w\right\rfloor-\lambda  \mu  u\left\lfloor vw\right\rfloor-\lambda  \nu  uvw\\\vspace{1ex}
		&-\lambda  \mu \left\lfloor uv\right\rfloor w-\lambda  \mu  uv\left\lfloor w\right\rfloor+\mu ^2\left\lfloor uvw\right\rfloor+\mu  \nu  uvw-\nu  uv\left\lfloor w\right\rfloor\\\vspace{1ex}
		=&\lambda \left\lfloor u\right\rfloor v\left\lfloor w\right\rfloor+\lambda ^2 u\left\lfloor v\right\rfloor w-\lambda  \mu  u\left\lfloor vw\right\rfloor\\\vspace{1ex}
		&-\lambda  \mu \left\lfloor uv\right\rfloor w+\mu ^2\left\lfloor uvw\right\rfloor-(\lambda -\mu ) \nu  uvw,
	\end{array}	$$
	where the last equation holds since the coefficient of $uv\left\lfloor w\right\rfloor$ is $\lambda ^{2}-\lambda  \mu -\nu $, which vanishes  by definition.
	
	On the other hand,
	$$	\begin{array}{lll}\vspace{1ex}
		&\alpha R(\phi(v,x))\\\vspace{1ex}
		=&\lambda \left\lfloor u\right\rfloor\left\lfloor v\right\rfloor w+\lambda  \left\lfloor u\right\rfloor v\left\lfloor w\right\rfloor-\mu \left\lfloor u\right\rfloor\left\lfloor vw\right\rfloor-\nu \left\lfloor u\right\rfloor vw\\\vspace{1ex}
		\rightarrow&\lambda  (\lambda \left\lfloor u\right\rfloor v+\lambda  u\left\lfloor v\right\rfloor-\mu \left\lfloor uv\right\rfloor-\nu  uv)w+\lambda  \left\lfloor u\right\rfloor v\left\lfloor w\right\rfloor\\\vspace{1ex}
		&-\mu  (\lambda \left\lfloor u\right\rfloor vw+\lambda  u\left\lfloor vw\right\rfloor-\mu \left\lfloor uvw\right\rfloor-\nu  uvw)-\nu \left\lfloor u\right\rfloor vw\\\vspace{1ex}
		=&\lambda ^2\left\lfloor u\right\rfloor vw+\lambda ^2 u\left\lfloor v\right\rfloor w-\lambda  \mu \left\lfloor uv\right\rfloor w-\lambda  \nu  uvw+\lambda  \left\lfloor u\right\rfloor v\left\lfloor w\right\rfloor\\\vspace{1ex}
		&-\lambda  \mu \left\lfloor u\right\rfloor vw-\lambda  \mu  u\left\lfloor vw\right\rfloor+\mu ^2\left\lfloor uvw\right\rfloor+\mu  \nu  uvw-\nu \left\lfloor u\right\rfloor vw\\\vspace{1ex}
		=&\lambda \left\lfloor u\right\rfloor v\left\lfloor w\right\rfloor+\lambda ^2 u\left\lfloor v\right\rfloor w-\lambda  \mu  u\left\lfloor vw\right\rfloor\\\vspace{1ex}
		&-\lambda  \mu \left\lfloor uv\right\rfloor w+\mu ^2\left\lfloor uvw\right\rfloor-(\lambda -\mu ) \nu  uvw\quad (\text{by~} \lambda ^{2}=\lambda  \mu +\nu ).
	\end{array}	$$
	We have shown that  $R(\phi(u,v)) \gamma  \downarrow_{\Pi_\phi(Z)} \alpha R(\phi(v,x))$. For the second condition in Theorem~\ref{Thm: verify a GS OPI}, assume that there are $q\neq \star   \in \mathfrak{M}^{\star}(Z)$ and $u_1, u_2, v_1, v_2 \in\mathfrak{M}(Z)$, such that $\lfloor u_1\rfloor\lfloor u_2\rfloor=q|_{\lfloor v_1\rfloor\lfloor v_2\rfloor}$, then $\lfloor v_1\rfloor\lfloor v_2\rfloor$ is a subword of $u_1$ or $u_2$.

We are done.
\end{proof}

With a similar proof, we have the following result:
\begin{lem}\label{Lem: differential type 2a is GS}
	The OPIs of type $(\calu2'_a)$ in Theorem~\ref{Lemma: equivalent OPIs of unital diff type} are unital operated GS with respect to the monomial order $\leq_{\rm{db}}$.
\end{lem}

Before dealing with OPI of type $(\calu1'_b)=(\calu5'_a)$, we introduce  the following lemma.

\begin{lem}\label{cor: Nonunital differential OPI is gronber basis in  free unital operated algebra}
	Let $\phi\in \bk\frakS(X)$ be a   differential OPI of type $(\calu1'_b)$  in Theorem~\ref{Lemma: equivalent OPIs of unital diff type}.  Then $S_{\phi}(\frakM(Z)\backslash\{1\})$ is an operated
	GS  in $\bk\frakM(Z)$ with respect to $\leq_{\mathrm{udl}}$.
\end{lem}

\begin{proof}
Replacing $Z$ by $Z\sqcup  \dagger$ in Theorem~\ref{Lemma: Nonunital differential OPI is gronber basis in  free nonunital operated algebra},     $S_{\phi}(\frakS(Z\sqcup  \dagger))$ is an operated
GS basis  in $\bk\frakS(Z\sqcup  \dagger)$ with respect to $\leq_{\mathrm{dl}}$. Then by the definition of the order $\leq_{\mathrm{udl}}$  on $\frakM(Z)$, which  is induced by $\leq_{\mathrm{dl}}$ on $\frakM(Z)\backslash\{1\}=\frakS(Z\sqcup  \dagger)$, the assertion follows.
\end{proof}

We will not distinguish $\dagger$ from $\lfloor1\rfloor$ in the following lemma.

\begin{lem}\label{Lem: differential type 1b is GS}
	The OPIs of type $(\calu1'_b)=(\calu5'_a)$ in Theorem~\ref{Lemma: equivalent OPIs of unital diff type} are unital operated GS with respect to the monomial order $\leq_{\rm{udl}}$.
%	\begin{itemize}
%	\item [(a)] 		The OPI of type (1) with $b=0$ in Example~\ref{Ex: list of OPIs of nonunital diff  type} is GS  with respect to the monomial order $\leq_{\rm{dl}}$.
%	\item [(b)] The OPI of type (1b) in Theorem~\ref{Lemma: equivalent OPIs of unital diff type} is GS  with respect to the monomial order $\leq_{\rm{udl}}$.
%	\end{itemize}
\end{lem}

\begin{proof}
%	(a) With a similar discussion in Theorem~\ref{Lemma: equivalent OPIs of unital diff type} in the case $b=0$, type (1) in Example~\ref{Ex: list of OPIs of nonunital diff  type}  is equivalent to type (1b) or (1c) in Theorem~\ref{Lemma: equivalent OPIs of unital diff type} considered as OPIs for nonunital case. It is easy for type (1c). And for type (1b), i.e. $x\lfloor y\rfloor-(\lfloor xy\rfloor-\lfloor x\rfloor y+cx y)$, let $R(x,y)=\lfloor xy\rfloor-\lfloor x\rfloor y+cx y$. By Lemma~\ref{Lemma: Nonunital differential OPI is gronber basis in  free nonunital operated algebra}, it suffices to show $R(uv,w)\downarrow_{\phi}uR(v,w)$.
	
%	On one hand,
%		$$	\begin{array}{lll}\vspace{1ex}
%		&R(uv,w) \\\vspace{1ex}
%		=&\lfloor uvw\rfloor-\lfloor uv\rfloor w+cuvw.
%	\end{array}	$$

%	On the other hand,
%$$	\begin{array}{lll}\vspace{1ex}
%	&uR(v,w) \\\vspace{1ex}
%	=&u\lfloor vw\rfloor-u\lfloor v\rfloor w+cuvw\\\vspace{1ex}
%	\rightarrow&\lfloor uvw\rfloor-\lfloor u\rfloor vw+cuvw-(\lfloor uv\rfloor w-\lfloor u\rfloor vw+cuvw)+cuvw\\\vspace{1ex}
%		=&\lfloor uvw\rfloor-\lfloor uv\rfloor w+cuvw.
%\end{array}	$$

Let $\phi(x,y)=x\lfloor y\rfloor-R(x,y)$ be the OPI of type $(\calu1'_b)=(\calu5'_a)$ and denote $Z^\prime=Z\cup \lfloor 1\rfloor$.

Firstly, we will show that $ S_{\phi}(\frakS(Z^\prime))\cup\{  \lfloor 1\rfloor-\lambda \} $ is an operated GS basis in $\bk\frakM(Z)$ with respect to $\leq_{\rm{udl}}$. By Lemma~\ref{cor: Nonunital differential OPI is gronber basis in  free unital operated algebra}, $S_{\phi}(\frakS(Z^\prime))$ is an operated GS basis in $\bk\frakM(Z)$.
So we only need to consider the inclusion composition of  $f\in S_{\phi}(\frakS(Z^\prime))$ and $g=\lfloor 1\rfloor-\lambda$. Let  $w=\bar{f}=u\lfloor v\rfloor=q|_{\lfloor1\rfloor}$ where $u,v\in\frakS(Z^\prime)$, then $q$ has a left factor $u$ or a right factor $\lfloor v\rfloor$, i.e., $q=u\lfloor q'\rfloor$ or $q=q'\lfloor v\rfloor$ with $q'\in \frakM^\star(Z)$.

 For the first case, we have $v=q'|_{\lfloor 1\rfloor}$, and the inclusion composition
$$\begin{aligned}
	\left( f,g\right)^q_{w}&=f-q|_g\\
	&=\phi(u,q'|_{\lfloor 1\rfloor}) -u\lfloor q'|_{\lfloor1\rfloor}\rfloor+u\lfloor q'|_{\lambda}\rfloor\\
	&=-R(u,q'|_{\lfloor 1\rfloor-\lambda})+\lambda\left( u\lfloor q'|_{1}\rfloor- R(u,q'|_{1})   \right)\\
	&=-R(u,q'|_{\lfloor 1\rfloor-\lambda})+\lambda\phi(u,q'|_{1}).
\end{aligned}$$
If $q'$ has a factor in $\cals(Z)$, then $\phi(u,q'|_{1})$ is trivial modulo $(S_{\phi}(\frakS(Z^\prime))  ,w)$; otherwise we have  $q'=\star$, thus $\phi(u,q'|_{1})= u\lfloor 1\rfloor-\lambda u$ is trivial modulo $( \{  \lfloor 1\rfloor-\lambda \} ,w)$.  And since  $R(u,q'|_{\lfloor 1\rfloor-\lambda})$ is trivial modulo $( \{  \lfloor 1\rfloor-\lambda \} ,w)$, it follows that $\left( f,g\right)^q_{w}$ is trivial modulo $(S_{\phi}(\frakS(Z^\prime))\cup\{  \lfloor 1\rfloor-\lambda \} ,w)$.

For the second case, we have $u=q'|_{\lfloor 1\rfloor}$, and similar to the first case, the inclusion composition
$$	\left( f,g\right)^q_{w}=-R(q'|_{\lfloor1\rfloor-\lambda},v)+\lambda\phi(q'|_{1},v)$$
is trivial modulo $(S_{\phi}(\frakS(Z^\prime))\cup\{  \lfloor 1\rfloor-\lambda \} ,w)$.
So we have shown that $ S_{\phi}(\frakS(Z^\prime))\cup\{  \lfloor 1\rfloor-\lambda \} $ is an operated GS basis in $\bk\frakM(Z)$.

Now we show that $S_{\phi}(\frakM(Z))$ is also an operated GS basis in $\bk\frakM(Z)$.
Since $\frakM(Z)=\frakS(Z^\prime)\cup\{1\}$, one can write
$$	\begin{aligned}
	S_{\phi}(\frakM(Z))
	&=\{\phi(u,v)\ |\ u,v\in \frakM(Z)\}\\
	&=\{\phi(u,v)\ |\ u,v\in \frakS(Z^\prime)\}\cup\{\phi(u,1), \phi(1,u), \phi(1,1)\ |\ u\in \frakS(Z^\prime)\}\\
	&=S_{\phi}(\frakS(Z^\prime))\cup\{u\lfloor 1\rfloor-\lambda u, \lfloor 1\rfloor u-\lambda u, \lfloor 1\rfloor-\lambda~|~u\in\frakS(Z^\prime)\}.
\end{aligned}	$$
So it is easy to see $ \langle S_{\phi}(\frakS(Z^\prime)) \cup\{  \lfloor 1\rfloor-\lambda \}\rangle_{\uOpAlg}=\langle S_{\phi}(\frakM(Z))\rangle_{\uOpAlg}$ and
$\Irr( S_{\phi}(\frakS(Z^\prime)) \cup\{  \lfloor 1\rfloor-\lambda \})=\Irr( S_{\phi}(\frakM(Z)))$.  By Theorem~\ref{Thm: unital CD}, the set $ \Irr( S_{\phi}(\frakS(Z^\prime))\cup\{  \lfloor 1\rfloor-\lambda \}) $  is a linear basis of the operated algebra $\bk\frakM(Z)/\langle S_{\phi}(\frakS(Z^\prime)) \cup\{  \lfloor 1\rfloor-\lambda \}\rangle_{\uOpAlg}$, i.e., $ \Irr( S_{\phi}(\frakM(Z)))$  is a linear basis of $\bk\frakM(Z)/\langle S_{\phi}(\frakM(Z))\rangle_{\uOpAlg}$.
By  Theorem~\ref{Thm: unital CD} again,
the set $S_{\phi}(\frakM(Z))$ is also an operated GS basis in $\bk\frakM(Z)$ with respect to $\leq_{\rm{udl}}$.
\end{proof}

To deal with OPI of type $(\calu5'_b)$, we need a unital version of Proposition~\ref{Proposition: Nonunital differential OPI xdy is GSB in  free nonunital operated algebra}, whose  proof is similar, so we omit it.
\begin{prop}\label{Proposition: unital differential OPI xdy is GSB in  free unital operated algebra}
	Let $\phi=x \lfloor y\rfloor -R(x,y) \in\bk \frakM(X )$ be an OPI. Let $Z$ be a set such that $\frakM (Z)$ is  endowed with a monomial order.  Assume that for all $u,v\in \frakM (Z)$,  $\phi(u,v) $   vanishes or its  leading monomial is still $u\lfloor v\rfloor$. Let $$\Pi_{\phi}(Z):= \{q|_{u\lfloor v\rfloor}\rightarrow q|_{R(u,v)}~|~u,v \in \frakM(Z),~ q\in\frakM^\star(Z)\}.$$ If
	$R(uv,w)\downarrow_{\Pi_{\phi}(Z)}uR(v,w)$, for any $u,v,w\in\frakM(Z)$, then $S_{\phi}(\frakM(Z))$ is an operated	GS basis in $\bk\frakM(Z)$ with respect to $\leq$.
\end{prop}

\begin{lem}\label{Lem: differential type 4b is GS}
	The OPI  of type $(\calu5'_b)$ in Theorem~\ref{Lemma: equivalent OPIs of unital diff type} is  unital operated GS with respect to the monomial order $\leq_{\rm{udl}}$.
\end{lem}

\begin{proof}
	For OPI $\phi=x\lfloor y\rfloor-(\lfloor xy\rfloor-\lfloor x\rfloor y+x\lfloor 1\rfloor  y)$ of type $(\calu5'_b)$,
	let $R(x,y)=\lfloor xy\rfloor-\lfloor x\rfloor y+x\lfloor 1\rfloor y$. Consider the rewriting system $$\Pi_{\phi}(Z):= \{q|_{u\lfloor v\rfloor}\rightarrow q|_{R(u,v)}~|~u,v \in \frakM(Z),~ q\in\frakM^\star(Z)\}.$$
	Firstly we show $ R(uv,w)\downarrow_{\Pi_\phi(Z)}uR(v,w)$.
	We have
	$$	\begin{array}{lll}\vspace{1ex}
		&uR(v,w) \\\vspace{1ex}
		=&u\lfloor vw\rfloor-u\lfloor v\rfloor w+ uv\lfloor 1\rfloor w\\\vspace{1ex}
		\rightarrow&\lfloor uvw\rfloor-\lfloor u\rfloor vw+ u\lfloor 1\rfloor vw-(\lfloor uv\rfloor w-\lfloor u\rfloor vw+ u\lfloor 1\rfloor vw)+ uv\lfloor 1\rfloor w\\\vspace{1ex}
		=&\lfloor uvw\rfloor-\lfloor uv\rfloor w+ uv\lfloor 1\rfloor w\\\vspace{1ex}
		=&R(uv,w).
	\end{array}	$$

If $u=1$ or $v=1$, then $\phi(u,v)=0$, and for $u,v\in \frakM(Z)\backslash 1$, $u\lfloor v\rfloor$ is always the leading monomial of $\phi(u,v)$ by Proposition~\ref{Lemma: inequal  relation}. It follows from Proposition~\ref{Proposition: unital differential OPI xdy is GSB in  free unital operated algebra} that $S_{\phi}(\frakM(Z))$ is an operated  GS basis in $\bk\frakM(Z)$.
\end{proof}

Notice that we can not apply Proposition~\ref{Proposition: unital differential OPI xdy is GSB in  free unital operated algebra} to the proof of Lemma~\ref{Lem: differential type 1b is GS}, since for the OPI $\phi$ of type $(\calu1'_b)=(\calu5'_a)$ in Theorem~\ref{Lemma: equivalent OPIs of unital diff type}, $\overline{\phi(u,v)}\neq u\lfloor v\rfloor$ if we take $u=1$ and $v\neq1$. In fact, $\overline{\phi(1,v)}=\lfloor 1\rfloor v\neq\lfloor v\rfloor$.

\begin{lem}\label{Lem: differential type with one variable is GS}
	Each OPI of the types $(\calu1'_c)=(\calu2'_b)=(\calu6'_b)$, $(\calu3')=(\calu4')=(\calu6'_a)$ and $(\calu5'_c)=(\calu6'_c)$ in Theorem~\ref{Lemma: equivalent OPIs of unital diff type} is unital operated GS with respect to the monomial order $\leq_{\rm{db}}$.
\end{lem}
\begin{proof}
 It is easy to see each of these OPIs satisfies the conditions in Theorem~\ref{Thm: verify a GS OPI}, so the assertion follows.
\end{proof}

We have considered all  OPI in Theorem~\ref{Lemma: equivalent OPIs of unital diff type}.
\begin{thm}\label{Thm: differential type are GS}
	Each OPI in Theorem~\ref{Lemma: equivalent OPIs of unital diff type} is unital operated GS with respect to either  the monomial order $\leq_{\rm{udl}}$ or $\leq_{\rm{db}}$.
\end{thm}

\subsection{Operated GS bases of free unital differential type algebras over algebras} \

By Theorem~\ref{Thm: differential type are GS}, Theorem~\ref{Thm: GS basis for free unital Phi algebra over unital alg} implies immediately the following result.

\begin{thm}\label{Prop: GS basis for unital differential type 1}
	Let $Z$ be a set and $A=\bk \calm(Z)\slash I_A$ a unital algebra with a GS  basis $G$ with respect to $\leq_{\rm{dlex}}$.	
		\begin{itemize}
		\item[(a)] Let $\phi_1$ be an OPI of type $(\calu1'_a)$ or $(\calu2'_a)$.
		Then $S_{\phi_1}(\frakM(Z))\cup G$ is an operated GS  basis of $\langle S_{\phi_1}(\frakM(Z))\cup I_A\rangle_\uOpAlg$ in $\bk\frakM(Z) $ with respect to $\leq_{\rm{db}}$.
		% And the set
	%	$$\Irr(S_{\phi_1}(\frakM(Z))\cup G)=\frakM(Z)\backslash \left\lbrace q|_{\bar{s}},q|_{\lfloor u\rfloor\lfloor v\rfloor}~|~s\in G,q\in\frakM^\star(Z),u,v\in\frakM(Z)\right\rbrace$$ is a basis of the free $\phi_1$-algebra $\calf^{\phi_1\zhx\uAlg}_{\uAlg}(A)$ over $A$.
		
	\item[(b)] Let $\phi_2$ be an OPI of type $(\calu1'_b)=(\calu5'_a)$ and $(\calu5'_b)$.
	Then $S_{\phi_2}(\frakM(Z))\cup G$ is an operated GS  basis of $\langle S_{\phi_2}(\frakM(Z))\cup I_A\rangle_\uOpAlg$ in $\bk\frakM(Z) $ with respect to $\leq_{\rm{udl}}$.

	\item[(c)] Let $\phi_3$ be an OPI of type  $(\calu3')=(\calu4')=(\calu6'_a)$.
	Then $S_{\phi_3}(\frakM(Z))\cup G$ is an operated GS  basis of $\langle S_{\phi_3}(\frakM(Z))\cup I_A\rangle_\uOpAlg$ in $\bk\frakM(Z) $ with respect to $\leq_{\rm{db}}$.

\item[(d)]   Let $\phi_4$ be an OPI of type $(\calu1'_c)=(\calu2'_b)=(\calu6'_b)$. Then the free  unital $\phi_4$-algebra $\calf^{\phi_4\zhx\uAlg}_{\uAlg}(A)$ is isomorphic to $A$.
	% And the set
	%$$\Irr(S_{\phi_3}(\frakM(Z))\cup G)=\calm(Z)\backslash \left\lbrace q|_{\bar{s}}~|~s\in G\right\rbrace$$
	%=\frakM(Z)\backslash \left\lbrace q|_{\bar{s}},q|_{\lfloor u\rfloor}~|~s\in G,q\in\frakM^\star(Z),u\in\frakM(Z)\right\rbrace is a basis of the free $\phi_3$-algebra $\calf^{\phi_3\zhx\uAlg}_{\uAlg}(A)$ over $A$.

	\item[(e)] Let $\phi_5$ be an OPI of type $(\calu5'_c)=(\calu6'_c)$.
	Then the free  unital $\phi_5$-algebra  $\calf^{\phi_5\zhx\uAlg}_{\uAlg}(A)$ vanishes.
	\end{itemize}
\end{thm}

By Theorem~\ref{Thm: unital CD}, we obtain linear bases for these algebras with a simple proof.
\begin{thm}\label{thm: linear basis of unital differential type algebra over algebra}
	With the same setup as in Theorem~\ref{thm: GS basis for nonunital differential type} then we have the following statements:
	\begin{itemize}
		\item[(a)]
		Let $B_0=\Irr_{\calm}(G)$ and for any $n\geq 1$, define the set $B_{n}$  inductively as follows
	$$\bigcup_{r\geq0}\left\lbrace {u_{0}\lfloor v_{1}\rfloor u_{1}    \cdots\lfloor v_{r} \rfloor u_{r}~|~u_0,u_r\in\Irr_{\calm}(G) , u_i\in\Irr_{\cals}(G),1\leq i\leq r-1,  v_j\in B_{n-1},1\leq j\leq r }\right\rbrace.
	$$
		Then the set
	$B= \bigcup_{n\geq0}B_n$ is a  linear basis of the free  unital $\phi_1$-algebra $\calf^{\phi_1\zhx\uAlg}_{\uAlg}(A)$ over the unital algebra $A$.
		
		\item[(b)] The set
		$$\bigcup_{r\geq0}\left\lbrace {\lfloor\cdots\lfloor\lfloor u_{r}\rfloor^{(k_r)} u_{r-1}\rfloor^{(k_{r-1})} \cdots u_{1}\rfloor^{(k_1)}  u_{0}~\Big|~u_{0}, u_i\in\Irr_{\calm}(G), k_i\geq1, 1\leq i\leq r}\right\rbrace$$
		is a linear  basis of the free  unital $\phi_2$-algebra $\calf^{\phi_2\zhx\uAlg}_{\uAlg}(A)$ over $A$.
		
		\item[(c)] The set
	$$\bigcup_{r\geq0}\left\lbrace {u_{0}\lfloor 1\rfloor^{k_1} u_{1} \lfloor 1\rfloor^{k_2}   \cdots\lfloor1 \rfloor^{k_r} u_{r}~\Big|~u_0, u_r\in\Irr_{\calm}(G) , u_i\in \Irr_{\cals}(G),1\leq i\leq r-1, k_j\geq 1, 1\leq j\leq r}\right\rbrace
	$$
	%=\frakM(Z)\backslash \left\lbrace q|_{\bar{s}},q|_{\lfloor u\rfloor}~|~s\in G,q\in\frakM^\star(Z),u\in\frakM(Z)\right\rbrace
	is a  linear basis of the free  unital $\phi_3$-algebra $\calf^{\phi_3\zhx\uAlg}_{\uAlg}(A)$ over $A$.
	%	\item[(d)]
	%The free $\phi_5$-algebra over any unital algebra $A$ is $0$.	
	\end{itemize}		
\end{thm}

\begin{exam} By  taking $a=1$  and $c=0$  in
	   the  type ($\calu$1) of   Conjecture~\ref{Ex: list of OPIs of diff  type},  we get the    OPI:
	$$\phi_\Dif(x,y)=\left\lfloor xy\right\rfloor-x\left\lfloor y\right\rfloor-\left\lfloor x\right\rfloor y-b\left\lfloor x\right\rfloor\left\lfloor y\right\rfloor.$$
	A $\phi_\Dif$-algebra is called a differential algebra of weight $b$. Given a unital algebra $A$ with a GS basis $G$,
	then  we get a linear basis for  the free unital differential algebra of nonzero weight over $A$   by Theorem~\ref{thm: linear basis of unital differential type algebra over algebra} (a) and by Theorem~\ref{thm: linear basis of unital differential type algebra over algebra} (b)
in weight zero case.
\end{exam}

\begin{remark}\label{Rem: Guo and Li's work} In \cite{GuoLi21}, Guo and Li introduced the notion of differential GS bases in order to study  free differential algebras over algebras.
	   They    showed that  the free differential algebra $\calf_{\uAlg}^{\phi_\Dif\zhx\uAlg}(A)$ over a unital $\bk$-algebra $A=\bk\calm(X)\slash I_A$ is the same as  the free differential algebra over the set $X$ modulo the differential ideal generated by the ideal $I_A$ \cite[Proposition 2.7]{GuoLi21}, which can be deduced from  Proposition~\ref{Prop: from algebra to Phi algebra with relations} as well. They also gave  a linear basis of $\calf_{\uAlg}^{\phi_\Dif\zhx\uAlg}(A)$    \cite[Proposition 3.2]{GuoLi21}.

Our method is completely different from theirs.
	
\end{remark}

\bigskip

 \textbf{Acknowledgements:}   The authors were  supported   by  NSFC (No. 12071137) and by STCSM (No. 18dz2271000).

The authors would like to express their sincere gratitude to Xing Gao, Li Guo, Yunnan Li, Shanghua Zheng for giving useful comments.

\end{document}